\documentclass[a4paper,11pt]{amsart}

\usepackage[left=2cm,top=3cm,bottom=2.5cm,right=2.5cm]{geometry}

\usepackage{amsthm}
\usepackage{amssymb}
\usepackage{amsmath}
\usepackage{mathtools}
\usepackage{pdfpages}
\usepackage{changes}
\usepackage{exscale,cite,color,amsopn}
\usepackage[colorlinks=true,pdftex,unicode=true,linktocpage,bookmarksopen,hypertexnames=true]{hyperref}
\usepackage{aliascnt}
\usepackage{csquotes}
\usepackage{colortbl}
\usepackage{enumitem}
\usepackage{verbatim}
\usepackage{hyperref}
\usepackage{tikz-cd}
\usepackage[capitalise]{cleveref}

\usetikzlibrary{trees}

\renewcommand{\leq}{\leqslant}
\renewcommand{\geq}{\geqslant}
\DeclareMathOperator{\Fix}{Fix}
\DeclareMathOperator{\id}{id}

\DeclareMathOperator{\Ret}{Ret}
\DeclareMathOperator{\Soc}{Soc}
\DeclareMathOperator{\ret}{ret}

\DeclareMathOperator{\mpl}{mpl}

\newcommand{\Z}{\mathbb{Z}}

\newcommand{\Zen}{\mathcal{Z}}
\newcommand{\Aut}{\operatorname{Aut}}

\newcommand{\Hol}{\operatorname{Hol}}

\newcommand{\Sym}{\mathfrak{S}}
\newcommand{\Endl}{\mathrm{End}_{\lambda}}

\newcommand{\Orb}{\mathcal{O}}

\newcommand{\G}{\mathcal{G}}

\newcommand{\gen}[1]{\left\langle #1 \right\rangle}
\newcommand{\genrel}[2]{\left\langle \ #1 \ \vline \ #2 \ \right\rangle}

\newcommand{\cab}[1]{\underset{#1}{\ast}}
\newcommand{\cabsim}[1]{\underset{#1}{\sim}}

\makeatletter
\numberwithin{equation}{section}
\numberwithin{figure}{section}
\numberwithin{table}{section}
\newtheorem{thm}{Theorem}[section]
\newtheorem*{thm*}{Theorem}
\newtheorem{lem}[thm]{Lemma}
\newtheorem{cor}[thm]{Corollary}
\newtheorem{pro}[thm]{Proposition}

\theoremstyle{definition}
\newtheorem{defn}[thm]{Definition}
\newtheorem*{defn*}{Definition}
\newtheorem{problem}[thm]{Problem}

\newtheorem*{convention*}{Convention}
\newtheorem{rem}[thm]{Remark}
\newtheorem{exa}[thm]{Example}

\makeatother

\title[Endocabling of solutions / Solutions whose diagonal is a cyclic permutation]{Endocabling of involutive solutions to the Yang-Baxter equation, with an application to solutions whose diagonal is a cyclic permutation}
\author{Carsten Dietzel}
\date{\today}

\address[Carsten Dietzel]{Normandie Univ, UNICAEN, CNRS, LMNO, 14000 Caen, France}
\email{carsten.dietzel@unicaen.fr}

\begin{document}

\begin{abstract}
In this article, we introduce endocabling as a technique to deform involutive, non-degenerate set-theoretic solutions to the Yang-Baxter equation (``solutions'', for short) by means of $\lambda$-endomorphisms of their associated permutation brace, thus generalizing the cabling method by Lebed, Vendramin and Ram\'{i}rez. In the first part of the article, we define endocabling and investigate the behaviour of solutions and their invariants under endocabling. In the second part, we apply our findings to solutions of size $n$ whose diagonal map is an $n$-cycle: we will prove that solutions with this property whose size is an odd prime power, are of finite multipermutation level. Furthermore, solutions with this property whose size is a power of $2$, will be proven either to be of finite multipermutation level or to admit an iterated retraction onto a unique solution of size $4$.

We formulate our results in the language of cycle sets.
\end{abstract}

\maketitle

\section*{Introduction}

This work is concerned with the \emph{Yang--Baxter equation}, a special case of the \emph{quantum Yang--Baxter equation} (\emph{QYBE}, for short) that has been discovered by Rodney James Baxter \cite{Baxter_YB} and Chen-Ning Yang \cite{Yang_YB} in the context of integrable systems. In its parameter-independent and braided form, a solution to the QYBE is given by a finite-dimensional vector space $V$, typically over the field of complex numbers, together with an invertible linear map $R: V^{\otimes 2} \to V^{\otimes 2}$ that satisfies the braiding condition
\begin{equation} \label{eq:QYBE}
    (\id_V \otimes R)(R \otimes \id_V)(\id_V \otimes R) = (R \otimes \id_V)(\id_V \otimes R)(R \otimes \id_V). \tag{QYBE}
\end{equation}
Historically, the theory of \cref{eq:QYBE} grew with the theory of quantized algebras and has been investigated mainly in the the context of mathematical physics. However, in the 80's, the observation of Turaev \cite{Turaev_links} that its braided nature can be exploited in order to provide novel knot invariants drew the attention of mathematicians outside mathematical physics to \cref{eq:QYBE}, as it turned out that it opened up the possibility of constructing generalizations of knot invariants such as the Jones polynomial \cite{jones_polynomial}. A main difficulty, however, was the fact that at that time, most known solutions of \cref{eq:QYBE} came from quantum groups resp. their representations. Furthermore, as most quantum groups occurred as deformations of ``classical'' commutative or cocommutative Hopf algebras, these solutions often were parametrical deformations of the trivial \emph{flip} solution $R(v \otimes w) = w \otimes v$.

In the 90's, Drinfeld \cite{Drinfeld_Problems} addressed this problem by motivating the study of \emph{set-theoretic} solutions to the Yang--Baxter equation which are defined as those solutions that permute the basis vectors $e_i \otimes e_j$ of $V^{\otimes 2}$ for some basis $(e_i)$ of $V$. These solutions could then be subjected to a deformation process in order to construct novel parametrized families of solutions. For set-theoretic solutions defined by racks, a systematic approach has, for example, been developed by Eisermann \cite{Eisermann_deformations}.

The set-theoretic counterpart to a solution to \cref{eq:QYBE} is given by a set $X$, together with a bijective map $r: X^2 \to X^2$ that satisfies the \emph{set-theoretic} Yang--Baxter equation
\begin{equation} \label{eq:YBE}
    (\id_X \times r)(r \times \id_X)(\id_X \times r)=(r \times \id_X)(\id_X \times r)(r \times \id_X) \tag{YBE}.
\end{equation}
An important feature of \cref{eq:YBE} is that special classes of its solutions can be effciently studied by means of algebraic tools, predominantly finite and discrete group theory. This discovery is one of the main achievements of the landmark articles of Etingof, Schedler and Soloviev \cite{ESS_YangBaxter}, Gateva-Ivanova and Van den Bergh \cite{GIVdB_IType} and Lu, Yan and Zhu \cite{LYZ_YangBaxter} and initiated the algebraic theory of the Yang--Baxter equation.

Writing a solution as $r(x,y) = (\lambda_x(y),\rho_y(x))$, it is usual to investigate solutions that are \emph{non-degenerate}, meaning that the maps $\lambda_x,\rho_x$ are bijective for all $x \in X$. Furthermore, solutions are often demanded to be \emph{involutive}, meaning that $r^2 = \id_{X^2}$. In this case, one can define a \emph{structure group}
\[
G(X,r) = \genrel{e_x, \ x\in X}{e_xe_y = e_we_z \textnormal{ if } r(x,y) = r(w,z)}_{\mathrm{gr}}
\]
and a \emph{permutation group}
\[
\G(X,r) = \gen{\lambda_x : x \in X} \leq \Sym_X
\]
which can be shown to be an epimorphic image of $G(X,r)$. These groups can be equipped with a bijective $1$-cocycle to a respective abelian group, a fact that has been conceptualized by Rump \cite{Rump_Decomposition, Rump_braces} by the notion of \emph{brace} (see \cref{sec:preliminaries} for details). This established brace theory as one of the main tools to investigate involutive, non-degenerate solutions.

Due to work by Guarnieri and Vendramin \cite{Guarnieri_Vendramin}, there now exists a brace-theoretic framework for non-involutive solutions, which relies on the notion of \emph{skew left brace} and includes Rump's braces under the notion of \emph{skew left brace of abelian type}. Furthermore, several efforts have been made to set up an algebraic framework for solutions that are not non-degenerate \cite{left_non_degenerate} or not even bijective \cite{finite_idempotent,stanovsky_idempotent}.

Coming back to finite, involutive, non-degenerate solutions (which we from now on simply refer to as \emph{solutions}), an important invariant is their \emph{diagonal} which is defined as the map $T: X \to X$; $x \mapsto \lambda_x^{-1}(x)$. This map can be shown to be bijective and captures several properties of $X$, in particular such that are related to decomposability, as has first been observed by Gateva-Ivanova. Here, a solution is called \emph{decomposable} if there is a decomposition $X = X_1 \sqcup X_2$ with $X_i \neq \emptyset$ ($i=1,2$) such that $r(X_i^2) = X_i^2$ ($i=1,2$). Below, we list some results of this kind:

\begin{enumerate}
    \item If $T = \id_X$, then $X$ is decomposable \cite{Rump_Decomposition}.
    \item If $T$ is an $n$-cycle, where $n = |X|$, then $X$ is indecomposable \cite{Vendramin_Ramirez_Decomposition}.
    \item If $T$ is an $n-1$-cycle, where $n = |X|$, then $X$ is decomposable \cite{Vendramin_Ramirez_Decomposition}.
    \item If the order $o(T)$ is coprime to $|X|$, then $X$ is decomposable \cite{camp_mora_sastriques}.
    \item Let $X$ be indecomposable and $|X| = pq$ where $p < q$ are prime, then either the length of each cycle of $T$ is disivible by $q$ or $T$ contains a cycle of length divisible by $p$ \cite{LRV_Cabling}. 
\end{enumerate}

However, the diagonal (resp. its type) of a solution is an invariant that is hard to control in order to relate it to properties of a solution. An important method that allows some leverage in this aspect is the \emph{cabling} method which has been developed by Lebed, Ram\'{i}rez and Vendramin \cite{LRV_Cabling} as a far-reaching generalization of ideas of Dehornoy in the context of his RC-calculus of solutions \cite{Dehornoy_RC}. Together with Rump's theorem on the decomposability of squarefree solutions \cite{Rump_Decomposition}, cabling is a strong tool for the investigation of decomposability of solutions with a given diagonal structure.

The main idea behind the cabling method is to deform a given solution $(X,r)$ with $r(x,y) = (\lambda_x(y),\rho_y(x))$ to a solution $(X,r^{(k)})$ for some fixed positive integer $k$. The $\lambda$-map of the thus deformed solution is given by
\[
\lambda^{(k)}_x(y) = \lambda_{k \cdot \lambda_x}(y) = (\lambda_x  \circ \lambda_{T(x)} \circ \ldots \circ \lambda_{T^{k-1}(x)} )(y).
\]
where $k \cdot \lambda_x$ is taken in the \emph{permutation brace} on $\G(X,r)$ (see \cref{sec:preliminaries} for details).

This method leads to easy proofs of the above-mentioned decomposability results of Ram\'{i}rez and Vendramin, and Camp-Mora and Sastriques by means of a reduction to Rump's decomposition theorem.

From then, the cabling method has been the basis for several new discoveries concerning set-theoretic solutions and has been - explicitly or implicitly - applied in \cite{Feingesicht,CO_SquarefreeIndecomposable}, for example. Very recently, in the work of Colazzo and Van Antwerpen \cite{colazzo_van_antwerpen}, a non-involutive generalization of cabling has been introduced in order to investigate decomposability of solutions in a non-involutive setting.

In this work, we show that a more general notion of cabling can be obtained by deforming solutions via \emph{$\lambda$-endomorphisms}, that is, homomorphisms $\varphi: \G(X,r)^+ \to \G(X,r)^+$ with respect to the brace structure such that $\lambda_g \circ \varphi = \varphi \circ \lambda_g$ for all $g \in \G(X,r)$. Given a $\lambda$-endomorphism $\varphi$, we show that $\lambda^{(\varphi)}_x(y) = \lambda_{\varphi(\lambda_x)}(y)$ is the $\lambda$-map of an involutive solution.

Furthermore, we prove that the diagonal maps are well-behaved under addition of $\lambda$-endo\-mor\-phisms. More precisely, we show that $T_{\varphi + \psi} = T_{\varphi} \circ T_{\psi}$ whenever $\varphi, \psi$ are two $\lambda$-endomorphisms of $\G(X,r)$ for a solution $X$. Therefore, one of the main features of endocabling, which is control over the diagonal, is also characteristic for the endocabling method. Furthermore, the fact that all diagonals obtained from an endocabling of the same solution commute places a lot of restrictions on the obtainable diagonals when $T$ has a small centralizer. This is, for example, the case, if $T$ is a full $n$-cycle where $n = |X|$.

We first investigate the notion of (weak) $\lambda$-endomorphisms of braces and describe the endocabling of a solution $X$ by (weak) $\lambda$-endomorphisms of the permutation brace $\G(X,r)$. Furthermore, we investigate the behaviour of the diagonal under sums of $\lambda$-endomorphisms and particular endocablings defined by central elements of $\G(X,r)$ (\cref{sec:endocabling}). We then briefly investigate irreducible solutions, that is, solutions without proper sub-solutions, in order to provide a framework for the investigation of solutions of size $n$ whose diagonal is an $n$-cycle (\cref{sec:irreducibility}). These solutions will then be investigated in the special case when $n = p^v$ where $p$ is prime: we prove that if $p$ is odd, then each such solution is of finite multipermutation level (see \cref{sec:preliminaries} for a definition) which will result from an analysis of the center of the permutation group via endocabling (\cref{sec:full_cycle_odd_prime_power}). Furthermore, we will show that each solution of size $n=2^v$ whose diagonal is an $n$-cycle is either of finite multipermutation level or its iterated retractions become stationary on a unique solution of size $4$. In the inductive proof, we will use more advanced properties of cablings that involve, for example, that central elements can be eliminated by means of endocablings (\cref{sec:full_cycle_power_of_two}).

We formulate our results in the language of cycle sets.

\section{Preliminaries} \label{sec:preliminaries}

\subsection{Group-theoretic notation}~

In this subsection, we list some group-theoretic notations and conventions that will be used throughout this work. Suppose, for now, that $G = (G,\circ)$ is a group with identity $e = e_G$.

\begin{enumerate}
    \item We write $H \leq G$ if $H$ is a subgroup of $G$. For a subset $S \subseteq G$, we denote the subgroup generated by $S$ by
    \[
    \gen{S}_{\circ} = \bigcap_{\substack{H \leq G, \\ S \subseteq H}} H.
    \]
    If the operation $\circ$ is clear from the context, we drop the subscript-$\circ$.
    \item The \emph{center} of $G$ is denoted by $\Zen(G) = \{ z \in G \ : \ \forall g \in G: z \circ g = g \circ z \}$. Furthermore, the \emph{centralizer} of a subset $S \subseteq G$ is denoted by $C_G(S) = \{g \in G \ : \ \forall s \in S: g \circ s = s \circ g \}$.
    \item The \emph{order} of an element $g \in G$ is denoted by $o_{\circ}(g) = |\gen{g}_{\circ}|$. The \emph{exponent} of $G$ is denoted by
    \[
    \exp_{\circ}(G) = \begin{cases}
        \min \{ n \geq 1 \ : \ \forall g \in G: g^n = e \} & \exists n \geq 1: \forall g \in G: g^n = e, \\
        \infty & \textnormal{else}.
    \end{cases}
    \]
    Alternatively, $\exp_{\circ}(G) = \mathrm{lcm} \{ o_{\circ}(g): g \in G \}$, if it is defined. If the operation $\circ$ is clear from the context, we drop the subscript-$\circ$.
    \item For a set $X$, we denote the \emph{symmetric group} over $X$ by
    \[
    \Sym_X = \{ \pi : X \to X : X \textnormal{ bijective} \}.
    \]
    We suppose that $\Sym_X$ acts \emph{from the left} on $X$. Furthermore, for an integer $n \geq 1$, we write $\Sym_n = \Sym_{\{1, \ldots,n \}}$.
    \item Given a group $G$ that acts from the left on a set $X$ by an operation $G \times X \to X$; $(g,x) \mapsto g \cdot x$, we denote the \emph{stabilizer} of an element $x_0 \in X$ by $G_{x_0} = \{ g \in G: g \cdot x_0 = x_0 \} \leq G$. Furthermore, we denote the \emph{orbit} of $x_0$ under the action of $G$ by $\Orb_G(x_0) = \{ g \cdot x_0 : g \in G \}$. We will use this notation in particular when $G$ is a \emph{permutation group}, that is, a subgroup of $\Sym_X$ for some set $X$.
\end{enumerate}

\medskip

\subsection{Cycle sets}~

We will rely on the language of \emph{cycle sets}, as defined by Rump \cite{Rump_Decomposition}. These are defined as follows:

\begin{defn}[{\cite[§4]{Rump_Decomposition}}]
    A \emph{cycle set} is a set $X$ with a binary operation $(x,y) \mapsto x \ast y$ such that the following axioms are satisfied:
    \begin{align}
        \sigma_x: X \to X ; \ & y \mapsto x \ast y \quad \textnormal{is bijective for all } x \in X, \tag{C1} \label{eq:C1_bijectivity} \\ 
        (x \ast y) \ast (x \ast z) & = (y \ast x) \ast (y \ast z) \quad \forall x,y,z \in X \tag{C2}. \label{eq:C2_cycloid_equation}
    \end{align}
    Furthermore, a cycle set is called \emph{nondegenerate} if
    \begin{equation}
        T: X \to X; x \mapsto x \ast x \quad \textnormal{is bijective}. \tag{C3} \label{eq:C3_nondegeneracy}
    \end{equation}
    The map $T$ is referred to as the \emph{diagonal map} of $X$.
\end{defn}

\emph{Homomorphisms} of cycle sets are defined as mappings that respect the $\ast$-operation. \emph{Isomorphisms} are, as usual, defined as bijective homomorphisms. A \emph{sub-cycle set} of a cycle set $X$ is defined as a subset of $X$ that is a cycle set by restriction of operations.

Note that nondegenerate cycle sets on a set $X$ are in bijective correspondence to nondegenerate, involutive, set-theoretic solutions to the Yang--Baxter equation on the set $X$. We explicitly repeat the correspondence established in \cite[Proposition 1]{Rump_Decomposition}:

Given a nondegenerate cycle set $(X,\ast)$, the associated set-theoretic solution $r: X^2 \to X^2$ is given by
\[
r(x,y) = (\sigma_x^{-1}(y),(\sigma_x^{-1}(y)) \ast x).
\]
On the other hand, given a nondegenerate, involutive, set-theoretic solution $r: X^2 \to X^2$; $(x,y) \mapsto (\lambda_x(y),\rho_y(x))$, the set $X$ becomes a cycle set under the binary operation $x \ast y = \lambda_x^{-1}(y)$.

Due to this transfer principle, we will from now on use the language of cycle sets to talk about set-theoretic solutions.

\begin{convention*}
    When writing \emph{cycle set}, we will always mean: \emph{a finite, non-degenerate cycle set}.
\end{convention*}

Note that by a theorem of Rump \cite[Theorem 2]{Rump_Decomposition}
for a finite cycle set, \eqref{eq:C3_nondegeneracy} follows from \eqref{eq:C1_bijectivity} and \cref{eq:C2_cycloid_equation}.

With each cycle set $X$ comes a group-theoretic invariant, its \emph{permutation group} \cite{Rump_Decomposition} which is defined as the group
\[
\genrel{\sigma_x}{x \in X} \leq \Sym_X.
\]
Given an element $g \in \Sym_X$, we denote its explicit action on the elements of $X$ by $x \mapsto \lambda_g(x)$.

 A cycle set is called \emph{indecomposable} if for each disjoint decomposition $X = X_1 \sqcup X_2$ into sub-cycle sets $X_1,X_2$, we have $X_1 = X$ or $X_2 = X$. Note that a cycle set is indecomposable if and only if $\G(X)$ acts transitively on $X$. Also note that the \emph{decomposition classes} of $X$, that is, the orbits of $X$ under the action of $\G(X)$, are sub-cycle sets by restriction.

The following property of indecomposable cycle set has been discovered by Cedó and Okniński:

\begin{pro} \label{pro:surjections_of_indecomposable_cycle_sets}
    Let $X,Y$ be cycle sets where $Y$ is indecomposable and let $f: Y \twoheadrightarrow X$ be a surjective homomorphism of cycle sets. Then $|f^{-1}(x)| = |f^{-1}(x')|$ for all $x,x' \in X$. In particular, $|X|$ divides $|Y|$.
\end{pro}

\begin{proof}
    See \cite[Lemma 3.3]{CO_SquarefreeIndecomposable}.
\end{proof}

An central property of some cycle sets is their proneness to have repeated permutations. This is usually conceptualized by the notion of retractability. In the following, we collect definitions and results from \cite[§3.2]{ESS_YangBaxter} and \cite{Rump_Decomposition}.

Given a cycle set $X$, one defines the \emph{retraction relation} by $x \sim y \Leftrightarrow \sigma_x = \sigma_y$ which is immediately seen to be an equivalence relation. Accordingly, the equivalence classes are denoted as \emph{retraction classes}. It can be shown \cite[Lemma 2]{Rump_Decomposition} that $\sim$ is a congruence relation. This gives rise to the following notion:

\begin{defn}
    The \emph{retraction} of a cycle set $X$ is the set $X_{\ret} = X/\!\!\sim$, where $x \sim y \Leftrightarrow \sigma_x = \sigma_y$, under the cycle set operation given by $[x] \ast [y] = [x \ast y]$.
\end{defn}

Note that a cycle set is called \emph{retractable} if $\sim$ is a nontrivial equvalence relation which for finite cycle sets is the same as saying that $|X_{\ret}| < |X|$. Otherwise, a cycle set is called \emph{irretractable}.

It often happens that the retraction of a cycle set is itself retractable. Therefore, one recursively defines for a cycle set $X$ the sequence of cycle sets $X^{(n)}$ ($n \geq 0$) by $X^{(0)} = X$ and $X^{(n+1)} = X^{(n)}_{\ret}$. It can happen that $X^{(n)}$ is a cycle set with one single element, therefore giving rise to the notion of \emph{multipermutation level} of a cycle set, which is the quantity defined by
\[
\mpl(X) = \begin{cases}
    \min \{ k \in \Z_{\geq 0}: |X^{(k)}| = 1  \}  & \exists k \in \Z_{\geq 0} :  |X^{(k)}| = 1, \\
    \infty & \textnormal{else}.
\end{cases}
\]

\medskip

\subsection{Brace theory}~ \label{subsec:braces}

Brace theory is an important tool in the study of set-theoretic solutions resp. cycle sets. As in this work, we are only concerned with involutive solutions, we will only make use of \emph{skew left braces of abelian type}. In order to keep things short, we refer to these as \emph{braces}. For most of the definitions and results, we refer to the article of Cedó, Jespers and Okniński \cite{CJO_Braces}. For those that are not mentioned there, we provide another reference or a proof.

\begin{defn}
    A \emph{brace} is a triple $(B,+,\circ)$, where $(B,+)$, $(B,\circ)$ are groups - its \emph{additive} and \emph{multiplicative} group, respectively - such that $(B,+)$ is abelian and the equation
    \begin{equation} \label{eq:brace_equation}
        a \circ (b + c) = (a \circ b) - a + (a \circ c)
    \end{equation}
    is satisfied for all $a,b,c \in B$.
\end{defn}

As is usual, we will refer to a brace $(B,+,\circ)$ by its underlying set $B$. Furthermore, we write $B^+ = (B,+)$ and $B^{\circ} = (B,\circ)$. Note that the multiplicative identity $1$ of $B^{\circ}$ and the additive identity $0$ of $B^+$ coincide, i.e. $0=1$.

A \emph{homomorphism} of braces is defined in the obvious way, i.e. as a map that is a group homomorphism between the respective additive and multiplicative groups. Furthermore, given a brace $B$, a subset $A \subseteq B$ is a \emph{subbrace} of $B$ if $A$ is a subgroup of both $A^+$ and $A^{\circ}$ or, equivalently, if $A$ is a brace by restriction of the operations on $B$.

\begin{convention*}
    When doing calculations in braces, we demand that, as is usual, multiplication binds stronger than addition, that is, one evaluates $a \circ b + c = (a \circ b) + c$.

    Furthermore, we will also use the symbol $\circ$ for the composition of mappings. It will, however, always be clear from the context if we refer to the composition of mappings or to the multiplication in a brace. The only case where these notations collide is the multiplication in the permutation brace $\G(X)$ where this superficial ambiguity does not pose any risk of confusion as the brace multiplication coincides with the composition of mappings.
\end{convention*}

An important feature of braces is that their additive group is a module over the multiplicative group:

\begin{defn}
    For a brace $B$, the \emph{$\lambda$-action} is the map
    \[
    \lambda: B \times B \to  B; \ (g,a) \mapsto \lambda_g(a) = g \circ a - g.
    \]
\end{defn}

It can be shown that in a brace $B$, the identities
\begin{align}
    \lambda_g(a + b) & = \lambda_g(a) + \lambda_g(b), \\
    \lambda_g \circ \lambda_h & = \lambda_{g \circ h}
\end{align}
are satisfied for all $g,h,a,b$, or, stated in other words, the map $B^{\circ} \to \Aut(B^+)$; $g \mapsto \lambda_g$ is a well-defined group homomorphism.

Note that each brace defines a cycle set:

\begin{thm} \label{thm:cycle_set_on_a_brace}
    Let $B$ be a brace, then $B$ is a cycle set under the operation $g \ast h = \lambda_g^{-1}(h)$.
\end{thm}

Given a brace $B$, we will call a subgroup $I \leq B^+$ a \emph{left ideal} if $\lambda_g(I) \subseteq I$ for all $g \in B$. In this case, $I$ is a subbrace of $B$. Furthermore, an \emph{ideal} of a brace $B$ is a left ideal $I \subseteq B$ such that $I$ is normal in $B^{\circ}$.

Two important (left) ideals of a brace are given by its socle and its fix:

The \emph{socle} of a brace $B$ is defined as
\[
\Soc(B) = \{ g \in B \ : \ \forall h \in B: g + h = g \circ h \} = \{ g \in B \ : \ \forall h \in B: \lambda_g(h) = h \} = \ker(\lambda).
\]
It can be shown that $\Soc(B)$ is an ideal in $B$.

On the other hand, the \emph{fix} of a brace $B$ is defined as:
\[
\Fix(B) = \{ h \in B \ : \ \forall g \in B: g + h = g \circ h \} = \{ h \in B \ : \ \forall g \in B: \lambda_g(h) = h \}.
\]
$\Fix(B)$ is a left ideal in $B$ but not necessarily an ideal.

It is notable that the elements of the fix provide automorphisms of a brace:

\begin{pro}\label{pro:fix_of_brace_provides_automorphisms}
    Let $B$ be a brace and $f \in \Fix(B)$, then $\lambda_f(g) = {}^{\circ f}g = f \circ g \circ f^{\circ -1}$ for all $g \in B$. In particular, $\lambda_f \in \Aut(B)$.
\end{pro}

\begin{proof}
    See \cite[Lemma 1.1]{drz_psquare}.
\end{proof}

For a positive integer $n$, we denote by $\pi(n)$ the set of prime numbers dividing $n$. Given this function, we recall some important properties and subsets of braces related to divisibility.

Let $\pi$ be an arbitrary set of primes. Given such a set and a brace $B$, we can define the \emph{$\pi$-primary component} $B_{\pi} = \{ g \in B: \pi(o_+(g)) \subseteq \pi \}$, that is, the Hall-$\pi$ subgroup of $B^+$. As this is a characteristic subgroup of $B^+$, it is invariant under the $\lambda$-action and therefore, $B_{\pi}$ is a left ideal in $B$. In the case of a single prime $p$, i.e. $\pi = \{ p \}$, we drop the set notation and simply denote this left ideal by $B_p$, calling it the \emph{$p$-primary component}.

A central role will be played by what is called a \emph{$p$-brace}:

\begin{defn} \label{defn:p-brace}
    A \emph{$p$-brace} is a brace $B$ such that $\pi(|B|) = \{p \}$ for a prime $p$.
\end{defn}

Note that if the prime $p$ is specified, we will substitute it into this notion and talk about $2$-braces, $3$-braces,\ldots

\begin{pro} \label{pro:p_braces_have_nontrivial_fix}
    Let $B$ be a $p$-brace, then $\Fix(B) \neq 0$.
\end{pro}

\begin{proof}
    By definition, $\Fix(B)$ is the set of elements fixed under the $\lambda$-action of $B^{\circ}$ on $B^+$. It is well-known that an action of a finite $p$-group by automorphisms of a finite $p$-group has nontrivial fixed points, therefore the proposition follows.
\end{proof}

We will very often make use of the following argument that relates fix, center and socle of a brace. From now on, we abbreviate $\Zen(B^{\circ}) = \Zen(B)$.

\begin{pro} \label{pro:center_and_fix_intersect_in_socle}
    If $B$ is a brace, then $\Zen(B) \cap \Fix(B) \subseteq \Soc(B)$.
\end{pro}

\begin{proof}
    Let $g \in \Zen(B) \cap \Fix(B)$, then for all $h \in B$, we have $g \circ h = h \circ g = h + g$.
\end{proof}

\medskip

\subsection{The permutation brace of a cycle set}~

The strength of brace theory in the theory of cycle sets lies in the fact that there is a brace structure on $\G(X)$ that reflects a significant amount of the structure of $X$ as a cycle set.

For a cycle set $X$, a unique additive group operation $(g,h) \mapsto g + h$ can be defined on $\G(X)$ such that $\G(X)$ is a brace with $\lambda_x + \lambda_y = \lambda_x \circ \lambda_{x \ast y}$ for all $x,y \in X$ where $\circ$ is the usual composition of elements in $\G(X) \leq \Sym_X$ \cite{Rump_braces}.
If we want to emphasize the brace structure of $\G(X)$, we will also refer to $\G(X)$ as the \emph{permutation brace} of $X$. A fundamental property of the permutation brace is that the $\lambda$-action of $\G(X)$ on the set $\{ \lambda_x: x \in X \}$ coincides with the action of $\G(X)$ induced on $X_{\ret}$. More precisely, all $g \in \G(X)$, $x \in X$ satisfy the equation
\begin{equation} \label{eq:fundamental_property_of_permutation_brace}
    \lambda_g(\lambda_x) = \lambda_{\lambda_g(x)}.
\end{equation}
The fundamental identity \eqref{eq:fundamental_property_of_permutation_brace} will be used so often that we will in general \emph{not} refer to it when applying it.

In particular, the map $\Ret: X \to \G(X)$; $x \mapsto \lambda_x$ is a homomorphism of cycle set if $\G(X)$ is equipped with the cycle set structure described in \cref{thm:cycle_set_on_a_brace}. Note that $\Ret(x) = \Ret(y)$ if and only if $x \sim y$ with respect to the retraction relation on $X$, so $\Ret$ induces an isomorphism between $X_{\ret}$ and $\Ret(X)$ given by $[x] \mapsto \lambda_x$.

An important feature of the permutation brace $\G(X)$ is that it detects retractability of $X$:

\begin{thm} \label{thm:socle_detects_retractability}
    Let $X$ be a cycle set and $g \in \Soc(\G(X))$, then $\lambda_{\lambda_g(x)} = \lambda_x$ for all $x \in X$. In particular, if $X$ is irretractable then $\Soc(\G(X)) = 0$.
\end{thm}

\begin{proof}
    \cite[Lemma 2.1]{BCJO_irretrsqfree}.
\end{proof}

For a cycle set, the \emph{Dehornoy class} is defined as $d(X) = \exp_+(\G(X))$. Note that this is not the original definition given by Dehornoy in \cite{Dehornoy_RC} but it can be proven \cite[Theorem G]{LRV_Cabling} that the definition used here is equivalent to the definition of \emph{class} given by Dehornoy.

Note that if $X$ is indecomposable, the set $\Ret(X) \subseteq \G(X)$ is a single orbit under the $\lambda$-action of $\G(X)$, so $d(X) = o_+(x)$ for any $x \in X$, as $\G(X) = \gen{\Ret(X)}_+$.

\begin{thm} \label{thm:oT_divides_dX}
    Let $X$ be a cycle set with diagonal $T$, then $o(T)$ divides $d(X)$.
\end{thm}

\begin{proof}
    See \cite[Proposition 2.9.(ii)]{Feingesicht}.
\end{proof}

For an indecomposable cycle set $X$, it happens often - but not always, see \cite{rump_primes_in_coverings} - that $\pi(|\G(X)|) = \pi(|X|)$. We conceptualize this phenomenon by the notion of \emph{$p$-type} and $\pi$-type:

\begin{defn}\label{def:p_type}
    Let $X$ be an indecomposable cycle set. If there is a prime $p$ such that $\pi(|X|) = \pi(|\G(X)|) = \{ p \}$, we call $X$ a cycle set of \emph{$p$-type}. If $\pi(|X|) = \pi(|\G(X)|)$, we call $X$ a cycle set of \emph{$\pi$-type}.
\end{defn}

In order to avoid confusion, we note that if the prime $p$ resp. the set of primes $\pi = \pi(|X|)$ is given, the definition can accordingly be specialized, that is, we can then speak about cycle sets of $2$-type resp. $\{2,3,7\}$-type, for example. If we only want to emphasize that such primes resp. sets of primes exist, we do not specialize $p$ resp. $\pi$.

\medskip

\subsection{More brace theory}~

In this subsection, we collect some brace-theoretic facts that will be useful later:

For an integer $k \geq 0$, we define the brace $B_k$ as $B_k^+ = \Z_{2^k}$ with the usual addition, together with the multiplication
\begin{equation} \label{eq:multiplication_on_B_k}
    a \circ b = a + b - 2ab,
\end{equation}
where the product $ab$ is to be understood as the product in the residue class ring $\Z_{2^k}$.

The braces $B_k$ are instances of \emph{primary cyclic braces}, as investigated by Rump \cite[Section 6]{Rump_Cyclic_Braces}.

We will later make use of the fact that central involutions in a brace always generate a subbrace isomorphic to some $B_k$:

\begin{pro} \label{pro:central_involutions_generate_Bk}
    Let $B$ be a finite brace and $z \in \Zen(B)$ with $z^{\circ 2} = 0$. Then $\gen{z}_+$ is a subbrace of $B$ such that $\gen{z}_+ \cong B_k$ for some integer $k$.
\end{pro}

\begin{proof}
    It follows from $z^{\circ 2} = 0$ that $\lambda_z(z) = -z$, and thus, $\lambda_z(az) = -az$ for all integers $a$. This implies
    \[
    z \circ az = z + \lambda_z(az) = z - az = (1-a)z.
    \]
    As $z$ is central, we can compute
    \[
    \lambda_{az}(z) = az \circ z - az = z \circ az - az = (1-a)z - az = (1-2a)z.
    \]
    which implies for integers $a,b$ that
    \[
    az \circ bz = az + \lambda_{az}(bz) = az + b \cdot \lambda_{az}(z) = az + b(1-2a)z = (a + b - 2ab)z
    \]
    This proves that $\gen{z}_+$ is closed under $\circ$. It is clear that $\gen{z}_+ \cong \Z_m$ for some $m$ and, after identifying $az$ with $a \in \Z_m$, it can be seen that $\circ$ is described by \cref{eq:multiplication_on_B_k}. It remains to prove that $m = 2^k$ for some $k$: note that the $\lambda$-action on $\gen{z}_+$ is given by $\lambda_{az}(bz) = (1-2a) \cdot bz$ which shows that $1-2a$ is a unit in $\Z_m$ for all integers $a$. Therefore, $m = 2^k$ for some $k$.
\end{proof}

\begin{pro} \label{pro:structure_of_bk}
    Let $k \geq 2$, then:
    \begin{enumerate} [label=\alph*)]
        \item $\gen{1}_{\circ} = \{0,1 \}$ and $\gen{2}_+$ are subgroups of $B_k^{\circ}$ such that
        \[
        B_k^{\circ} = \gen{1}_{\circ} \times \gen{2}_+.
        \]
        \item The set of elements of order $ \leq 2$ in $B_k^{\circ}$ is $\{ 0,1,2^{k-1},2^{k-1}+1\}$.
        \item $\gen{2}_+^\circ \cong \Z_{2^{k-1}}$.
    \end{enumerate}
\end{pro}

\begin{proof}
    \begin{enumerate}[label=\alph*)]
        \item It is a routine check that $\gen{1}_{\circ} = \{0,1 \}$. Also, it is quickly checked that $\gen{2}_+$ contains $0$ and is multiplicatively closed which implies, by finiteness, that $\gen{2}_+ \leq B^{\circ}_k$. Furthermore, $\gen{1}_{\circ} \cap \gen{2}_+ = \{ 0 \}$. As
        $
        |\gen{1}_{\circ}| \cdot |\gen{2}_+| = 2 \cdot 2^{k-1} = 2^k = |B_k|
        $, it follows that $B_k^{\circ} = \gen{1}_{\circ} \times \gen{2}_+$.
        \item Let $a \in B_k$ be such that $a^{\circ 2} = 0$. We then have $0 = a^{\circ 2} = 2a - a^2 = 2a \cdot (1-a)$. If $a$ is in $\gen{2}_+$, then $1-a$ is invertible, so $2a = 0$ resp. $a \in \{0,2^{k-1} \}$ in this case. As $1^{\circ 2} = 0$, and $(B_k : \gen{2}_+) = 2$, the only remaining element of order $2$ is $1 \circ 2^{k-1} = 1 + 2^{k-1}$.
        \item $\gen{2}_+^{\circ}$ is an abelian group of order $2^{k-1}$ and, by the last part of the lemma, contains two elements of order $\leq 2$. It follows that $\gen{2}_+^{\circ}$ is cyclic and hence, $\gen{2}_+ \cong \Z_{2^{k-1}}$.
    \end{enumerate}
\end{proof}

We will later need the following result:

\begin{pro} \label{pro:commutativity_of_cyclic_subbraces}
    Let $B$ be a brace and let $a,b \in B$. If $a$ centralizes $\gen{b}_+$ and $b$ centralizes $\gen{a}_+$, then $\gen{a}_+$ centralizes $\gen{b}_+$.
\end{pro}

\begin{proof}
    Let $a,b$ be as in the statement. For arbitrary integers $m,n$, we calculate:
    \begin{align*}
        ma \circ nb & = ma + n \lambda_{ma}(b) \\
        & = ma + n(ma \circ b - ma) \\
        & = ma + n(b \circ ma - ma) \\
        & = ma + n(b + m \lambda_b(a) - ma) \\
        & = ma + n(b + m(b \circ a - b) - ma ) \\
        & = (m - mn)a + (n-mn)b + mn(b \circ a).
    \end{align*}
    The last expression is symmetric in $ma$ and $nb$, so the statement follows.
\end{proof}

\section{Endocabling} \label{sec:endocabling}

The aim in this section is to investigate deformations of cycle sets by means of what we call \emph{$\lambda$-endomorphisms} of braces.

\begin{defn}
    Let $B$ be a brace. A homomorphism $\varphi: B^+ \to B^+$ is a \emph{$\lambda$-endomorphism} if for all $g,h \in B$, we have
    \begin{equation} \label{eq:endocabling_condition}
    \lambda_g(\varphi(h)) = \varphi(\lambda_g(h)).
    \end{equation}
    A homomorphism $\varphi: B^+ \to B^+$ is a \emph{relative $\lambda$-endomorphism} if \cref{eq:endocabling_condition} is satisfied for all $g \in \varphi(B)$ and $h \in B$.
\end{defn}

Denote by $\Endl(B)$ the set of $\lambda$-endomorphisms of a brace $B$ and by $\Endl'(B)$ the set of its relative $\lambda$-endomorphisms.

Using a module-theoretic language, a $\lambda$-endomorphism is an endomorphism of $B^+$, considered as a (left) $\Z[B^{\circ}]$-module by means of the $\lambda$-action.

\begin{exa}
Let $B$ be a brace.
    \begin{enumerate}
        \item For any $k \in \Z$, the map $\varphi: B \to B$; $g \mapsto k \cdot g$ is a $\lambda$-endomorphism of $B$.
        \item Let $z \in \Zen(B)$, then $\lambda_z \in \Endl(B)$ as $\lambda_z$ commutes with $\lambda_g$ for all $g \in B$.
        \item More generally, let $\Z[B^{\circ}]$ be the integral group ring of $B^{\circ}$, denote its center by $\Zen(\Z[B^{\circ}])$ and let
        \[
        r = \sum_{g \in B} k_g e_g \in \Zen(\Z[B^{\circ}]),
        \]
        where $e_g$ is the generator of $\Z[B^{\circ}]$ corresponding to the group element $g \in B^{\circ}$. Then
        \[
        \varphi: B \to B; \ h \mapsto \sum_{g \in B} k_g \lambda_g(h) 
        \]
        is a $\lambda$-endomorphism of $B$ because the center of a ring $R$ acts by endomorphisms on each $R$-module.
    \end{enumerate}
\end{exa}

A straightforward consequence of the definition of a $\lambda$-endomorphism is:

\begin{pro} \label{pro:basics_for_lambda_endomorphisms}
    Let $B$ be a brace, then:
    \begin{enumerate}[label=\alph*)]
        \item $\varphi \in \Endl(B) \Rightarrow \varphi \in \Endl'(B)$,
        \item $\varphi \circ \psi \in \Endl(B)$ for $\varphi,\psi \in \Endl(B)$,
        \item $\varphi \circ \psi \in \Endl'(B)$ for $\psi \in \Endl'(B)$, $\varphi \in \Endl'(\psi(B))$\footnote{Note that this is a slight abuse of notation: by $\varphi \circ \psi$, we mean the map that is constructed by composing $\varphi$ with the corestriction of $\psi$ to the image of $\psi$, and then extending the codomain of the composition to $B$ again!}.
        \item $\varphi + \psi \in \Endl(B)$ for $\varphi, \psi \in \Endl(B)$.
    \end{enumerate}
\end{pro}

As consequence, $\Endl(B)$ is a ring under addition and composition of endomorphisms that is unital with $1_{\Endl(B)} = \id_B$. From now on, we will often abbreviate $\id_B = \id$ for the identical $\lambda$-endomorphism, in favour of an economical notation when talking about $\lambda$-endomorphisms of a permutation brace $\G(X)$.

We will now introduce the technique of \emph{endocabling}: let $X$ be a cycle set with permutation brace $\G(X)$. Given (relative) $\lambda$-endomorphisms of $\G(X)$, it is possible to deform the operation of $X$ in the following way:

\begin{defn}
    Let $X = (X, \ast)$ be a cycle set and let $\varphi \in \Endl'(\G(X))$. Then the \emph{$\varphi$-cabling} of $X$ is the set $X_{\varphi} = X$, with the binary operation
    \[
    x \cab{\varphi} y = \lambda_{\varphi(\lambda_x)}^{-1}(y).
    \]
\end{defn}

\begin{pro}
    Let $X$ be a cycle set and $\varphi \in \Endl'(\G(X))$, then $X_{\varphi}$ is a cycle set.
\end{pro}

\begin{proof}
    As $\varphi(\lambda_x) \in \G(X)$ for all $x \in X$, cycle set axiom \eqref{eq:C1_bijectivity} is obvious. Now let $x,y,z \in X$. We calculate
    \begin{align*}
        (x \cab{\varphi} y) \cab{\varphi} (x \cab{\varphi} z) & = \lambda_{\varphi(\lambda_x)}^{-1}(y) \cab{\varphi} \lambda_{\varphi(\lambda_x)}^{-1}(z) \\
        & = \lambda^{-1}_{\varphi(\lambda_{\varphi(\lambda_x)}^{-1}(\lambda_y))} (\lambda_{\varphi(\lambda_x)}^{-1}(z)) \\
        & = \lambda^{-1}_{\varphi(\lambda_x) \circ \varphi(\lambda_{\varphi(\lambda_x)}^{-1}(\lambda_y))} (z) \\
        & = \lambda^{-1}_{\varphi(\lambda_x) \circ \lambda^{-1}_{\varphi(\lambda_x)}(\varphi(\lambda_y))}(z) \\
        & = \lambda^{-1}_{\varphi(\lambda_x) + \varphi(\lambda_y)}(z).
    \end{align*}
    As the final term is symmetric in the arguments $x,y$, the same follows for the first term, therefore proving the validity of the cycle set axiom \eqref{eq:C2_cycloid_equation} for $X_{\varphi}$. As $X$ is assumed to be finite, non-degeneracy \eqref{eq:C3_nondegeneracy}, follows from \cite[Theorem 2]{Rump_Decomposition}.
\end{proof}

\begin{rem}
    The author has been informed by Senne Trappeniers about the fact that weak endomorphisms give rise to deformations of skew braces in a similar way \cite[Proposition 6.18]{bi_skew_brace_blocks}.
\end{rem}

\begin{exa} \label{exa:some_endocablings}
We list some examples of endocablings:

    \begin{enumerate}
        \item Let $\varphi = k \cdot \id \in \Endl(\G(X))$ for some integer $k \geq 0$, then $X_{\varphi}$ coincides with the $k$-cabled cycle set $X^{[k]}$, as defined by Lebed, Ram\'{i}rez, Vendramin \cite{LRV_Cabling}.
        \item For $\varphi = -\id \in \Endl(\G(X))$, the cabling $X_{\varphi}$ is the \emph{dual} cycle set structure on $X$, in the sense of Rump \cite[Definition 1]{Rump_Decomposition}, with the operation $x \Tilde{\ast} y = \lambda_{T^{-1}(x)}(y)$.
        \item Let $X$ be a cycle set of $p$-type (see \cref{def:p_type}) and let $f \in \Fix(\G(X))$ be such that $f^{\circ p} = 0$. Then there is a unique $\varphi \in \Endl(\G(X))$ with $\varphi(\lambda_x) = f^{\circ -1}$ for all $x \in X$. The cycle set $X_{\varphi}$ then carries the operation
        $x \cab{\varphi} y = \lambda^{-1}_{\varphi(\lambda_x)}(y) = \lambda_{f^{\circ -1}}^{-1}(y) = \lambda_f(y)$.
    \end{enumerate}
\end{exa}

It turns out that we have a decent amount of control over the permutation group of $X$ when deforming $X$ by endocabling:

\begin{pro} \label{pro:permutation_group_of_endocabled_cycle_set}
    Let $X$ be a cycle set and let $\varphi \in \Endl(\G(X))$, then $\G(X_{\varphi}) = \varphi(\G(X))$. Furthermore, if $\varphi \in \Endl(\G(X))$, then $\G(X_{\varphi})$ is a left ideal in $\G(X)$.
\end{pro}

\begin{proof}
    By definition, $\G(X_{\varphi}) = \gen{\varphi(\lambda_x): x \in X}_{\circ}$, so it is trivial that $\G(X_{\varphi}) \leq \varphi(\G(X))$. For $g,h \in \G(X)$, we calculate
    \[
    \varphi(g) + \varphi(h) = \varphi(g) \circ \lambda^{-1}_{\varphi(g)}(\varphi(h)) = \varphi(g) \circ \varphi(\lambda^{-1}_{\varphi(g)}(h)).
    \]
    As $\{ \varphi(\lambda_x) : x \in X \} \subseteq \G(X_{\varphi})$, this shows that
    \[
    \varphi(\G(X)) = \varphi(\gen{\lambda_x : x \in X}_+)  = \gen{\varphi(\lambda_x): x \in X}_+ \subseteq \gen{\varphi(\lambda_x): x \in X}_{\circ} = \G(X_{\varphi}),
    \]
    thus proving $\G(X_{\varphi}) = \varphi(\G(X))$.

    The second statement follows from the first one, taking into account the observation that the image of a $\lambda$-endomorphism of a brace is a left ideal.
\end{proof}

It is now appropriate to explain why we decided to introduce the notion of \emph{relative} cabling: for $\varphi \in \Endl(\G(X))$, the brace $\G(X_{\varphi})$ is in general a proper subbrace of $\G(X)$, therefore with $\psi \in \Endl(\G(X_{\varphi}))$, the endocabled cycle set $(X_{\varphi})_{\psi}$ does not necessarily arise from a single endocabling by an element of $\Endl(\G(X))$. However, from \cref{pro:basics_for_lambda_endomorphisms}c) we see that $\psi \circ \varphi \in \Endl'(\G(X))$. So, while the notion of relative endomorphisms is not stable under addition of endomorphisms, it has the feature of being stable under iterated endocablings. With this, we can describe iterated endocablings as the result of a single relative endocabling, as in the proof of \cref{thm:full_diagonal_implies_retractable_2_type}.

One of the main reasons for the usefulness of ``classical'' cabling is that, under cabling, the diagonal behaves in a controllable way. We will soon see that endocabling is not too different with regard to this aspect.

We now demonstrate that two cablings of the same cycle set result in \emph{mixed} cycle sets in the sense of \cite[Proposition 3.8]{Feingesicht}:

\begin{pro}\label{pro:cablings_are_mixed}
    Let $X$ be a cycle set and let $\varphi, \psi \in \Endl(\G(X))$, then all $x,y,z \in X$ satisfy
    \begin{equation} \label{eq:cablings_are_mixed}
        (x \cab{\varphi} y) \cab{\psi} (x \cab{\varphi} z) = (y \cab{\psi} x) \cab{\varphi} (y \cab{\psi} z).
    \end{equation}
\end{pro}

\begin{proof}
    We calculate
    \begin{align*}
        \lambda_{\varphi(\lambda_x) + \psi(\lambda_y)}^{-1}(z) & = \lambda^{-1}_{\varphi(\lambda_x) \circ \lambda^{-1}_{\varphi(\lambda_x)}(\psi(\lambda_y))}(z) \\
        & = \left( \lambda^{-1}_{\lambda^{-1}_{\varphi(\lambda_x)}(\psi(\lambda_y))} \circ \lambda^{-1}_{\varphi(\lambda_x)} \right) (z) \\
        & = \lambda^{-1}_{\psi(\lambda^{-1}_{\varphi(\lambda_x)}(\lambda_y))}(x \cab{\varphi} z) \\
        & = \lambda^{-1}_{\psi \left(\lambda_{\lambda^{-1}_{\varphi(\lambda_x)}(y)} \right)}(x \cab{\varphi} z) \\
        & = \lambda^{-1}_{\psi(\lambda_{x \cab{\varphi} y})} (x \cab{\varphi} z) \\
        & = (x \cab{\varphi} y) \cab{\psi} (x \cab{\varphi} z).
    \end{align*}
    Similarly, one proves that
    \[
    \lambda_{\varphi(\lambda_x) + \psi(\lambda_y)}^{-1}(z) = \lambda_{\psi(\lambda_y) + \varphi(\lambda_x)}^{-1}(z) = (y \cab{\psi} x) \cab{\varphi} (y \cab{\psi} z),
    \]
    thus proving the proposition.
\end{proof}

In the following, we denote the diagonal of $X_{\varphi}$ by $T_{\varphi} = x \cab{\varphi} x$. A consequence of the above proposition is

\begin{pro} \label{pro:adding_two_cablings}
    Let $X$ be a cycle set and let $\varphi, \psi \in \Endl(\G(X))$, then
    \begin{equation} \label{eq:adding_two_cablings}
        x \cab{\varphi + \psi} y = T_{\varphi}(x) \cab{\psi} (x \cab{\varphi} y) = T_{\psi}(x) \cab{\varphi} (x \cab{\psi} y).
    \end{equation}
\end{pro}

\begin{proof}
From \cref{pro:cablings_are_mixed}, we deduce that
\[
(x \cab{\varphi} x) \cab{\psi} (x \cab{\varphi} y) = (x \cab{\psi} x) \cab{\varphi} (x \cab{\psi} y) \Rightarrow T_{\varphi}(x) \cab{\psi} (x \cab{\varphi} y) = T_{\psi}(x) \cab{\varphi} (x \cab{\psi} y).
\]
Furthermore, the calculation from \cref{pro:cablings_are_mixed} shows that
\[
x \cab{\varphi +\psi} y = \lambda_{(\varphi + \psi)(\lambda_x)}^{-1}(y) = \lambda_{\varphi(\lambda_x) + \psi(\lambda_x)}^{-1}(y) = (x \cab{\varphi} x) \cab{\psi} (x \cab{\varphi} y) = T_{\varphi}(x) \cab{\psi} (x \cab{\varphi} y).
\]
\end{proof}

 Specializing $y = x$ in \cref{eq:adding_two_cablings}, we get as a consequence:

\begin{pro} \label{pro:diagonals_are_homomorphisms}
    Let $X$ be a cycle set and $\varphi, \psi \in \Endl(\G(X))$. Then
    \begin{equation} \label{eq:diagonals_are_homomorphisms}
        T_{\varphi + \psi} = T_{\varphi} \circ T_{\psi} =  T_{\psi} \circ T_{\varphi}.
    \end{equation}
\end{pro}

\begin{proof}
    Using \cref{pro:adding_two_cablings}, we calculate for $x \in X$:
    \[
    T_{\varphi + \psi}(x) = (x \cab{\psi} x) \cab{\varphi} (x \cab{\psi} x) = (T_{\varphi} \circ T_{\psi})(x).
    \]
    As $\varphi + \psi = \psi + \varphi$, we see that $T_{\varphi}$ and $T_{\psi}$ commute.
\end{proof}

As a corollary, we recover the following result of Lebed, Ram\'{i}rez and Vendramin \cite[Theorem B.(1)]{LRV_Cabling}:

\begin{cor}
    Let $X$ be a cycle set with diagonal $T$, then $T_{k \cdot \id} = T^k$ for all $k \geq 0$.
\end{cor}

\begin{proof}
    $T_{k \cdot \id} = T_{\underbrace{\id + \ldots + \id}_{k \times}} = \underbrace{T_{\id} \circ \ldots \circ T_{\id}}_{k \times} = T^k$.  
\end{proof}

In the case, when $\varphi = \lambda_z$, where $z \in \Zen(\G(X))$, the diagonal of $X_{\lambda_z}$ is also easily computed:

\begin{pro} \label{pro:diagonal_of_cabling_by_center}
    Let $X$ be a cycle set and $z \in \Zen(\G(X))$, then
    \[
    T_{\lambda_z} = T^{\lambda_z} = \lambda_z^{-1} \circ T \circ \lambda_z.
    \]
\end{pro}

\begin{proof}
    \[
        T_{\lambda_z}(x)  = x \cab{\lambda_z} x 
         = \lambda^{-1}_{\lambda_z(\lambda_x)}(x) 
         = \lambda^{-1}_{\lambda_z(x)}(x) 
         = \lambda_z(x) \ast x 
         = \lambda_z^{-1} (\lambda_z(x) \ast \lambda_z(x)) 
         = \lambda_z^{-1} ( T (\lambda_z(x))) 
         = T^{\lambda_z}(x).
    \]
\end{proof}

In favour of a more economical notation, we will abbreviate $T^{\lambda_z} = T^z$.

The following corollary of \cref{pro:diagonal_of_cabling_by_center} will be of high importance for the next section, so we state it as a theorem:

\begin{thm} \label{thm:diagonal_commutes_with_conjugated_diagonal}
    Let $X$ be a cycle set and $z \in \Zen(\G(X))$, then $T \circ T^z = T^z \circ T$.
\end{thm}

\begin{proof}
    Using \cref{pro:diagonals_are_homomorphisms} and \cref{pro:diagonal_of_cabling_by_center}, we calculate
    \[
    T \circ T^z = T_{\id} \circ T_{\lambda_z} = T_{\id + \lambda_z} = T_{\lambda_z} \circ T_{\id} = T^z \circ T.
    \]
\end{proof}

For a brace $B$ and an element $z \in \Zen(B)$, an important role will later be played by the endomorphism $\id - \lambda_z \in \Endl(B)$. We collect some properties of this endomorphism that we will regularly refer to later:

\begin{pro} \label{pro:properties_of_phi_z}
    Let $B$ be a brace and $z \in \Zen(B)$ with $o_{\circ}(z) = 2$. Let $\varphi = \id - \lambda_z \in \Endl(B)$, then:
    \begin{enumerate}[label=\alph*)]
        \item $\varphi^2 = 2\varphi$. In particular, $\varphi(g) = 2g$ for all $g \in \varphi(B)$.
        \item $\lambda_z(g) = -g$ for all $g \in \varphi(B)$.
        \item $\lambda_g(z) = z-2g$ for all $g \in \varphi(B)$.
        \item $\varphi(g) = z - \lambda_g(z)$ for all $g \in B$.
    \end{enumerate}
\end{pro}

\begin{proof}
    a) and b) are immediate. In order to prove c), we let $g \in \varphi(B)$ and calculate
    \[
    \lambda_g(z) = g \circ z - g = z \circ g - g = z + \lambda_z(g) - g \overset{b)}{=} z- g -g = z-2g.
    \]
    For d), we calculate
    \[
    \varphi(g) = g - \lambda_z(g) = g - z \circ g + z = z - (g \circ z - g) = z - \lambda_g(z).
    \]
\end{proof}

One interesting consequence that will be of high importance later is the fact that, using endocabling, one can \emph{cable out the center}. We put this observation as an independent proposition:

\begin{pro} \label{pro:cabling_out_the_center}
    Let $B$ be a brace and $z \in \Zen(B)$. Let $\varphi = \id - \lambda_z \in \Endl(B)$, then $z \not \in \varphi(B)$.
\end{pro}

\begin{proof}
    Note that \cref{pro:properties_of_phi_z}d) is valid without the condition $o_{\circ}(z) = 2$. From this, it follows that for all $g \in B$, we have $\varphi(g) = z - \underbrace{\lambda_g(z)}_{\neq 0} \neq z$.
\end{proof}

Given a cycle set $X$ and an element $f \in \Fix(\G(X))$, the map $\lambda_f$ is an automorphism of $\G(X)$ by \cref{pro:fix_of_brace_provides_automorphisms}. We now prove that $\lambda_f$ is also an automorphism of the cycle set $X$:

\begin{pro} \label{pro:fix_provides_automorphisms}
    Let $X$ be a cycle set and let $f \in \Fix(\G(X))$, then $\lambda_f$ is an automorphism of $X$. Furthermore, $\lambda_f$ centralizes $T$, the diagonal of $X$.
\end{pro}

\begin{proof}
    Using \cref{pro:fix_of_brace_provides_automorphisms}, we calculate for all $x,y \in X$:
    \[
    \lambda_f(x \ast y) = (\lambda_f \circ \lambda_x^{-1})(y) = \lambda_{{}^{\circ f}\lambda_x^{-1}} (\lambda_f(y)) = \lambda_{\lambda_f(\lambda_x^{-1})}(\lambda_f(y)) = \lambda_{\lambda_f(x)}^{-1}(\lambda_f(y)) = \lambda_f(x) \ast \lambda_f(y).
    \]
    In particular, $T(\lambda_f(x)) = \lambda_f(x) \ast \lambda_f(x) = \lambda_f(x \ast x) = \lambda_f(T(x))$, thus showing the other statement of the proposition.
\end{proof}

Later in this work, it will be highly useful that endocabling is a way of obtaining block systems on $X$:

\begin{pro} \label{pro:blocks_of_phi}
    Let $X$ be a cycle set and $\varphi \in \Endl(\G(X))$, then the equivalence relation given by
    \begin{equation} \label{eq:block_relation_of_phi}
    x \cabsim{\varphi} y \Leftrightarrow \varphi(\lambda_x) = \varphi(\lambda_y)
    \end{equation}
    is invariant under $\G(X)$, that is, for all $x,y \in X$, $g \in \G(X)$, the implication
    \[
    x \cabsim{\varphi} y \Rightarrow \lambda_g(x) \cabsim{\varphi} \lambda_g(y)
    \]
    holds.
\end{pro}

\begin{proof}
    If $x \cabsim{\varphi} y$, then
    \[
    \varphi(\lambda_{\lambda_g(x)}) = \varphi(\lambda_g(\lambda_x)) = \lambda_g(\varphi(\lambda_x)) = \lambda_g(\varphi(\lambda_y)) = \varphi(\lambda_g(\lambda_y)) = \varphi(\lambda_{\lambda_g(y)})
    \]
    which implies $\lambda_g(x) \cabsim{\varphi} \lambda_g(y)$.
\end{proof}

An (almost) immediate consequence of the definition of $\cabsim{\varphi}$ is the following:

\begin{pro}\label{pro:blocks_of_phi_are_retraction_classes}
    Let $X$ be a cycle set and $\varphi \in \Endl'(\G(X))$, then the equivalence classes of $X$ under $\cabsim{\varphi}$ are the retraction classes of $X_{\varphi}$.
\end{pro}

\begin{proof}
    Given $x,y \in X$, we have the equivalences
    \[
    x \cabsim{\varphi} y \Leftrightarrow \varphi(\lambda_x) = \varphi(\lambda_y) \Leftrightarrow \forall z \in X: \lambda^{-1}_{\varphi(\lambda_x)}(z) =  \lambda^{-1}_{\varphi(\lambda_y)}(z) \Leftrightarrow \forall z \in X: x \cab{\varphi} z = y \cab{\varphi} z. 
    \]
\end{proof}

The following result is a useful consequence that will play a striking role later. For example, it will allow us in proofs to replace central elements in $\G(X)$ by different central elements.

\begin{pro} \label{pro:replacement_of_center}
    Let $X$ be a cycle set such that $\G(X)$ is a $2$-group and let $z \in \Zen(\G(X))$ be such that $o_{\circ}(z) = 2$. Let $k \geq 1$ and $\varphi = \id - \lambda_z \in \Endl(\G(X))$. If $X_{k\varphi}$ is retractable and $d(X_{k\varphi})$  is divisible by $4$, then there is an element $z' \in \Zen(\G(X))$ with $z' \neq z$ and $o_{\circ}(z') = 2$.
\end{pro}

\begin{proof}
    Consider the equivalence relation $\cabsim{k\varphi}$ from \cref{pro:blocks_of_phi}. As the equivalence classes of $\cabsim{k\varphi}$ are the retraction classes of $X_{k\varphi}$, our assumption on $X_{k\varphi}$ implies that $\cabsim{k\varphi}$ is a nontrivial equivalence relation. By means of \cref{pro:blocks_of_phi}, we observe that there is a well-defined action of $\G(X)$ on $X/\!\!\cabsim{k\varphi}$ given by $g \cdot [x] = [\lambda_g(x)]$ for $g \in \G(X)$, $x \in X$, where $[x]$ denotes the equivalence class of $x \in X$ under $\cabsim{k\varphi}$. This action is obviously compatible with the projection $X \to X/\!\!\cabsim{k\varphi}$, so we get a group homomorphism $\chi: \G(X) \to \Sym_{X/\!\cabsim{k\varphi}}$; $ \ g \mapsto ([x] \mapsto [\lambda_g(x)])$. By \cref{thm:socle_detects_retractability}, we see that $\Soc(\G(X_{k\varphi})) \neq 0$ and for all $g \in \Soc(\G(X_{k\varphi}))$, we have $[\lambda_g(x)] = [x]$ by \cref{thm:socle_detects_retractability}. This shows that $\Soc(\G(X_{k\varphi})) \leq \ker(\chi)$. On the other hand, $z \not\in \ker(\chi)$: let $x \in X$ be such that $o_+(k\varphi(\lambda_x))$ is divisible by $4$ then, using \cref{pro:properties_of_phi_z}a), we calculate
    \[
    k\varphi(\lambda_x) - k\varphi(\lambda_{\lambda_z(x)}) = k\varphi(\lambda_x - \lambda_z(\lambda_x)) = k\varphi^2(\lambda_x) = 2k\varphi(\lambda_x) \neq 0.
    \]
    so $k\varphi(\lambda_x) \neq k\varphi(\lambda_{\lambda_z(x)})$, which shows that $x \not \cabsim{k\varphi} \lambda_z(x)$ and thus, $z \not\in \ker(\chi)$. But $\G(X_{k\varphi}) \leq \G(X)$ is a $2$-group and therefore nilpotent, so $\Zen(\G(X)) \cap \ker(\chi) \neq 0$, in particular, there is an element $z' \in \Zen(\G(X)) \cap \ker(\chi)$ with $o_{\circ} (z') = 2$. As $z \not\in \ker(\chi)$, it follows that $z' \neq z$ and we are done!
\end{proof}

\begin{rem}
    We chose a cabling method that relies on the notion of $\lambda$-endomorphisms because they are most convenient to work with in later applications. However, the reader may have noticed that plenty of proofs in this section just rely on the fact that for $\varphi \in \Endl(\G(X))$, the identity $\varphi \circ \lambda_x = \lambda_x \circ \varphi$ is satisfied for all $x \in X$.

    This gives rise to the following generalization: let $X$ be a cycle set and $\Phi: X \to \G(X)$ a \emph{$\lambda$-intertwiner}, that is, a mapping that satisfies
    \[
    \Phi(\lambda_x(y)) = \lambda_x(\Phi(y))
    \]
    for all $x,y \in X$. It is straightforward to show that $\Phi(\lambda_g(y)) = \lambda_g(\Phi(y))$ for each $\lambda$-intertwiner $\Phi: X \to \G(X)$. Given a $\lambda$-intertwiner $\Phi$, the operation on $X$ can be deformed to
    \[
    x \cab{\Phi} y = \lambda_{\Phi(x)}^{-1}(y),
    \]
    and it can again be shown that $X_{\Phi} = (X, \cab{\Phi})$ is a cycle set. Note that endocabling by an element $\varphi \in \Endl(\G(X))$ is the same as cabling by the $\lambda$-intertwiner $x \mapsto \varphi(\lambda_x)$.

    In particular, given $\lambda$-intertwiners $\Phi, \Psi: X \to \G(X)$, it can easily be shown that their pointwise sum $\Phi + \Psi$ is again a $\lambda$-intertwiner, and the same line of arguments leading to the proof of \cref{pro:diagonals_are_homomorphisms}, shows that with $T_{\Phi} = T_{X_{\Phi}}$, we again have
    \[
    T_{\Phi + \Psi} = T_{\Phi} \circ T_{\Psi}.
    \]

    Note that for an indecomposable cycle set $X$, there is a very convenient way to describe its $\lambda$-intertwiners: picking $x_0 \in \G(X)$, it is easy to show that for $g \in \G(X)$, a $\lambda$-intertwiner $\Phi$ with $\Phi(x_0) = g$ exists if and only if their stabilizers are related by $\G(X)_{x_0} \leq \G(X)_g$. In this case, $\Phi$ can be defined without conflict by $\Phi(\lambda_h(x_0)) = \lambda_h(g)$ for $h \in \G(X)$. Call the resulting $\lambda$-intertwiner $\Phi_{x_0,g}$. One can check that $\Phi_{x_0,g+h} = \Phi_{x_0,g} + \Phi_{x_0,h}$ for all eligible $g,h \in \G(X)$.

    The condition $\G(X)_{x_0} \leq \G(X)_g$ can more conveniently be stated as
    \[
    g \in \mathcal{F}_{x_0} = \{ h \in \G(X) \ : \ \forall g \in \G(X)_{x_0}: \lambda_g(h) = h \},
    \]
    which is a subgroup of $\G(X)^+$. Therefore, cabling by intertwiners of indecomposable cycle sets provides a group homomorphism
    \[
    \mathcal{T}: \mathcal{F}_{x_0} \to \Sym_X; \ g \mapsto T_{\Phi_{x_0,g}}.
    \]
\end{rem}

\section{Irreducibility of cycle sets} \label{sec:irreducibility}

In this section, we want to investigate the concept of \emph{irreducibility} of cycle sets that will be of high significance in our study of cycle sets of size $n$ whose diagonal is an $n$-cycle.

\begin{defn}\label{defn:irreducibility}
    A cycle set $X$ is \emph{irreducible} if $\emptyset$ and $X$ are its only sub-cycle sets.
\end{defn}

Note that irreducible cycle sets are necessarily indecomposable as the decomposition classes of a cycle set are always sub-cycle sets thereof.

An important observation is that irreducibility is inherited by epimorphic images:

\begin{pro} \label{pro:irreducibility_is_inherited_by_epimorphic_images}
    Let $f: X \twoheadrightarrow Y$ be a surjective homomorphism of cycle sets where $X$ is irreducible, then $Y$ is also irreducible.
\end{pro}

\begin{proof}
    Let $Z \subseteq Y$ be a nonempty sub-cycle set, then $f^{-1}(Z)$ is a sub-cycle set of $X$ which implies $f^{-1}(Z) = X$. As $f$ is surjective, this implies $Z = f(f^{-1}(Z)) = f(X) = Y$.
\end{proof}

We continue with an easy observation that is a slight generalization of an observation in \cite[Theorem 3.5]{Vendramin_Ramirez_Decomposition}

\begin{pro} \label{pro:full_diagonal_implies_no_sub_cyclesets}
    If $X$ is a cycle set of size $n$ such that its diagonal $T$ is an $n$-cycle, then $X$ is irreducible.
\end{pro}

\begin{proof}
    Suppose that $Y \subseteq X$ is a nonempty sub-cycle set. Let $x \in Y$, then $T(x) = x \ast x \in Y$, as $Y$ is a sub-cycle set. Inductively, this proves that $T^k(x) \in Y$ for all $k \geq 0$ but as $T$ is an $n$-cycle, we conclude that $Y = X$.
\end{proof}

For the proof of the next proposition, we will need the following lemma.

\begin{lem} \label{lem:dpt_lemma}
    Let $B$ be a brace and $L$ a left ideal, then $\Fix_B(L) = \{ h \in B \ : \ \forall g \in L:  \lambda_g(h) = h \}$ is a subbrace of $B$.
\end{lem}

\begin{proof}
    This is a special case of \cite[Lemma 4.1]{drz_psquare}.
\end{proof}

We are now in the position to prove

\begin{pro} \label{pro:irreducible_braces_are_of_pi_type}
    Let $X$ be an irreducible cycle set, then $X$ is of $\pi$-type.
\end{pro}

\begin{proof}
    As $\G(X)$ acts transitively on $X$, it is immediate that $\pi(|X|) \subseteq \pi(|\G(X)|)$. Therefore, in order to prove that $X$ is of $\pi$-type, it remains to show the inclusion $\pi(|\G(X)|) \subseteq \pi(|X|)$.
    
    Suppose that $q$ is a prime such that $q \nmid |X|$ and consider the left ideal $\G(X)_q$ Recall the definition of the homomorphism $\Ret: X \to \G(X)$; $x \mapsto \lambda_x$. By \cref{pro:irreducibility_is_inherited_by_epimorphic_images}, $\Ret(X)$ is irreducible as an epimorphic image of $X$. By \cref{lem:dpt_lemma}, $F = \Fix_{\G(X)}(\G(X)_q)$ is a subbrace of $\G(X)$ which implies that $\Ret(X) \cap F$ is a sub-cycle set of $\Ret(X)$. As $|\Ret(X)|$ divides $|X|$ (\cref{pro:surjections_of_indecomposable_cycle_sets}), we observe that $q \not\in \pi(|\Ret(X)|)$, which implies that $\G(X)_q$ has a fixed point on $\Ret(X)$, so $\Ret(X) \cap F$ is non-empty. As $\Ret(X)$ is irreducible, this implies that $\Ret(X) \cap F = \Ret(X)$ and thus, $\Ret(X) \subseteq F$. As $\G(X) = \gen{\Ret(X)}_+$, it follows that $\G(X)_q$ fixes all of $\G(X)$ which implies that $\G(X)_q \leq \Soc(\G(X))$. As $\G(X)_q$ is the Sylow-$q$ subgroup of $\Soc(\G(X))$, which is normal in $\G(X)$, we infer that $\G(X)_q$ is normal in $\G(X)$. But as $q \nmid |X|$, $\G(X)_q$ fixes at least one element of $X$ and, as $\G(X)_q$ is normal, it fixes all of $X$. This proves that $\G(X)_q = 1$ and thus, that $q \not\in \pi(|\G(X)|)$.
\end{proof}

Together with \cref{pro:full_diagonal_implies_no_sub_cyclesets}, this implies:

\begin{cor} \label{cor:full_n_cycle_implies_pi_type}
    Let $X$ be a cycle set of size $n$ whose diagonal $T$ is an $n$-cycle. Then $X$ is of $\pi$-type.
\end{cor}

\section{Cycle sets of odd prime-power size \texorpdfstring{$n$}{n} whose diagonal is an \texorpdfstring{$n$}{n}-cycle} \label{sec:full_cycle_odd_prime_power}

We will now make use of the endocabling machinery in order to gain insights into the structure of cycle sets of size $n$ whose diagonal is an $n$-cycle.

We first show that this property is preserved under epimorphic images:

\begin{pro} \label{pro:full_n_cycle_under_epimorphic_images}
Let $X$ be a cycle set of size $n$ whose diagonal $T_X$ is an $n$-cycle, and let $f: X \twoheadrightarrow Y$ be a surjection of cycle sets. If $m = |Y|$, then $m$ divides $n$ and $T_Y$ is an $m$-cycle.
\end{pro}

\begin{proof}
     By \cref{pro:surjections_of_indecomposable_cycle_sets}, the divisibility statement is clear. Now observe that as $f$ is a homomorphism, $f(T_X(x)) = T_Y(f(x))$ holds for all $x \in X$. Furthermore, $f(T_X^k(x)) = T_Y^k(f(x))$ holds for arbitrary integers $k$. In order to show that $T_Y$ is an $m$-cycle, it is sufficient to show that for all $y,y' \in Y$, there is an integer $k$ such that $T_Y^k(y) = y'$. Pick $x,x' \in X$ with $f(x) = y$, $f(x') = y'$, then there is a $k$ such that $T_X^k(x) = x'$. But then, $y' = f(x') = f(T_X^k(x)) = T_Y^k(f(x)) = T_Y^k(y)$, thus proving the claim.
\end{proof}

\begin{thm} \label{thm:full_diagonal_implies_finite_mpl}
    Let $X$ be a cycle set of size $n = p^v$ where $p$ is an \emph{odd} prime and $v \geq 1$. If its diagonal $T$ is an $n$-cycle, then $X$ is retractable.
\end{thm}

As a corollary, we deduce:

\begin{cor} \label{cor:full_diagonal_implies_finite_mpl}
    Let $X$ be a cycle set of size $n = p^v$ where $p$ is an \emph{odd} prime and $v \geq 0$. If its diagonal $T$ is an $n$-cycle, then $X$ is of finite multipermutation level.
\end{cor}

\begin{proof}
For $v = 0$, the statement is trivial. So suppose that $v > 0$ and that the statement is true for $v' < v$. Let $X$ be as in the statement of the corollary and consider the retraction homomorphism $\ret: X \twoheadrightarrow X_{\ret}$; $x \mapsto [x]$. By \cref{pro:full_n_cycle_under_epimorphic_images}, $|X_{\ret}| = p^{v'}$ for some $v' < v$ and $T_{X_{\ret}}$ is an $n'$-cycle where $n' = p^{v'}$. By the inductive assumption, $\mpl(X) = \mpl(X_{\ret})+1 < \infty$, thus proving the statement.
\end{proof}

We will make use of the following lemma. Recall that for a group $G$, its \emph{holomorph} $\Hol(G)$ is defined as the normalizer of $G$ in $\Sym_G$, where $G$ is identified with its image under the regular action. Note that $\Hol(G) \cong G \rtimes \Aut(G)$.

\begin{lem} \label{lem:fixed_point_free_elements_in_Hol(Zpv)_p_odd}
    Let $p$ be a prime and $v > 0$. Let $g \in \Hol(\Z_{p^v})$ be such that $g^p = 1$ and $g(x) \neq x$ for all $x$, then there is an $i$ with $0 < i < p$ such that $g(x) = x + i \cdot p^{v-1}$ for all $x \in \Z_{p^v}$.
\end{lem}

\begin{proof}
    Write $g(x) = \alpha x + \beta$ with $\alpha, \beta \in \Z_{p^v}$ where $\alpha$ is invertible in $\Z_{p^v}$. An standard inductive argument shows that for all $m \geq 0$, we have
    \begin{equation} \label{eq:power_of_f}
    g^m(x) = \alpha^m \cdot x + \left(\sum_{j=0}^{m-1} \alpha^j \right)\beta
    \end{equation}
    In particular, $\alpha$ has to satisfy $\alpha^p = 1$ which implies (\cite[Lemma 2.1.21]{cohen_nt}) that $\alpha = 1 + kp^{v-1}$ for some $0 \leq k < p$.

    If $\alpha = 1$, then $g(x) = x + \beta$ with $\beta \neq 0$. In this case, $g^p(x) = x + p\beta = x$, which implies $\beta = i \cdot p^{v-1}$ with $0 < i < p$.

    Now suppose that $\alpha = 1 + kp^{v-1}$ with $0 < k < p$. Replacing $g$ by an appropriate power, we may assume that $\alpha = 1 + p^{v-1}$. Another standard induction proves that
    \begin{equation} \label{eq:powers_of_alpha}
        (1 + p^{v-1})^j = 1 + jp^{v-1}
    \end{equation}
    for $j \geq 0$. With $\alpha = 1 + p^{v-1}$, we now calculate
    \[
    \sum_{j=0}^{p-1} \alpha^j = \sum_{j=0}^{p-1} 1+jp^{v-1} = p + \frac{p-1}{2}\cdot \underbrace{p p^{v-1}}_{=0} = p.
    \]
    Considering that $g^p = 1$, we plug this into \cref{eq:power_of_f} in order to obtain
    \[
    0 = \left(\sum_{j=0}^{m-1} \alpha^j \right)\beta = p \cdot \beta,
    \]
    which shows that $\beta = i \cdot p^{v-1}$ for some integer $i$. With this $i$, we get $g(x) = (1+p^{v-1})x + i \cdot p^{v-1}$. However, with this $g$, we also calculate
    \[
    g(-i) = (1+p^{v-1})(-i) + i \cdot p^{v-1} = -i,
    \]
    that is, $-i$ is a fixed point of $g$. We conclude that $\alpha = 1$, therefore proving the lemma.
\end{proof}

Furthermore, for reference, we recall the following well-known result about permutation groups. Recall that a permutation group $G \leq \Sym_X$ is \emph{semiregular} if the stabilizers are $G_x = \{ \id_X \}$ for all $x \in X$.

\begin{pro} \label{pro:centralizer_of_transitive_permutation_group_is_semiregular}
    Let $G \leq \Sym_X$ be a transitive permutation group, then $C_{\Sym_X}(G)$ acts semiregularly on $X$. In particular, if $g,h \in C_{\Sym_X}(G)$, then the equivalence $g \cdot x = h \cdot x \Leftrightarrow g = h$ holds for arbitrary $x \in X$.
\end{pro}

\begin{proof}
    Let $h \in C_{\Sym_X}(G)$ and $x \in X$ be such that $h \cdot x$. For any $y \in X$, pick $g \in G$ such that $g \cdot x = y$, then $h \cdot y = h \cdot (g \cdot x) = g \cdot (h \cdot x) = g \cdot x = y$. Therefore, $h = \id_X$. The second statement is an immediate consequence.
\end{proof}

We can now prove the main result of this subsection:

\begin{proof}[Proof of \cref{thm:full_diagonal_implies_finite_mpl}]
    Suppose that $X$ is irretractable. As its diagonal $T$ is an $n$-cycle, where $n=p^v$, we can write $X = \{ x_i : i \in \Z_{p^v} \}$ where $T(x_i) = x_{i+1}$. By \cref{cor:full_n_cycle_implies_pi_type}, we see that $\G(X)$ is a $p$-brace. As a consequence, $\Zen(\G(X)) \neq 0$ and $\Fix(\G(X)) \neq 0$ by \cref{pro:p_braces_have_nontrivial_fix}. Let $0 \neq f \in \Fix(\G(X))$ be such that $f^{\circ p} = 0$, then $T^f = T$ (\cref{pro:fix_provides_automorphisms}). As $T$ is an $n$-cycle, $\left\langle T \right\rangle$ is regular and abelian on $X$, so $\lambda_f \in \gen{T}$. As $\lambda_f \neq \id_X$, we can take a suitable power and assume that $\lambda_f(x_i) = x_{i + p^{v-1}}$.

    Now consider $0 \neq z \in \Zen(\G(X))$. By \cref{thm:diagonal_commutes_with_conjugated_diagonal}, we see that $T^z$ centralizes $T$ which, as before, implies $T^z \in \left\langle T \right\rangle$. Therefore, $\lambda_z$ normalizes $\gen{T}$ which implies that $\lambda_z(x_i) = x_{g(i)}$ for some $g \in \Hol(\Z_{p^v})$. Replacing $z$ by a suitable power, we can assume that $\lambda_z^p = \id_X$, and thus, $g^p = \id_X$. As a central element of a transitive permutation group, $\lambda_z$ does not have any fixed points on $X$ by \cref{pro:centralizer_of_transitive_permutation_group_is_semiregular}. As a consequence, $g$ is fixed-point-free. As $g^p = \id_X$, \cref{lem:fixed_point_free_elements_in_Hol(Zpv)_p_odd} now implies that $g(i) = i + k \cdot p^{v-1}$ with $0 < k < p$. Therefore, $\lambda_z(x_i) = x_{i+k \cdot p^{v-1}} = \lambda_f^k(x_i)$ for all $x_i \in X$. Therefore $z = f^{\circ k} \in \Fix(\G(X))$. By \cref{pro:center_and_fix_intersect_in_socle}, $z \in \Zen(\G(X)) \cap \Fix(\G(X)) \leq \Soc(\G(X))$, which by \cref{thm:socle_detects_retractability} implies that $X$ is retractable.
    \end{proof}

    \section{Cycle sets of size \texorpdfstring{$n=2^v$}{n=2v} whose diagonal is an \texorpdfstring{$n$}{n}-cycle} \label{sec:full_cycle_power_of_two}

    Note that \cref{thm:full_diagonal_implies_finite_mpl} fails to be true if $p=2$: in this case, there is a a unique irretractable cycle set of size $4$ such that $T$ is a $4$-cycle. In the GAP library \texttt{YangBaxter}, it is listed as \texttt{SmallCycleSet(4,19)}.
    Up to relabelling of elements, its multiplication table is given as follows:
    \begin{center}
        \begin{tabular}{c|c c c c}
        $\ast$ & 1 & 2 & 3 & 4 \\ \hline
        1 & 2 & 1 & 3 & 4 \\ 2 & 4 & 3 & 1 & 2 \\ 3 & 1 & 2 & 4 & 3 \\ 4 & 3 & 4 & 2 & 1
        \end{tabular}
    \end{center}
    Consequently, in the following, we will refer to this cycle set as $X_{4,19}$.

However, we can prove that \cref{thm:full_diagonal_implies_finite_mpl} is \emph{almost} true for $p=2$! Whereas \cref{thm:full_diagonal_implies_finite_mpl} only makes use of \cref{thm:diagonal_commutes_with_conjugated_diagonal} which can be proven with a little bit of effort without resorting to endocabling, we will use the full power of the endocabling machine in the proof of the following theorem:

\begin{thm} \label{thm:full_diagonal_implies_retractable_2_type}
    Let $X$ be a cycle set of size $n = 2^v$ where $v > 2$. If its diagonal $T$ is an $n$-cycle, then $X$ is retractable.
\end{thm}

By a similar argument as in the proof of \cref{cor:full_diagonal_implies_finite_mpl}, we obtain the following result:

\begin{cor}
    Let $X$ be a cycle set of size $n = 2^v$ where $v \geq 0$. If its diagonal $T$ is an $n$-cycle, then exactly one of the following cases occurs:
    \begin{enumerate}
        \item $\mpl(X) < \infty$,
        \item There is a $k \geq 0$ such that $X^{(k)} \cong X_{4,19}$.
    \end{enumerate}
\end{cor}

We will first prove a lemma that is similar to \cref{lem:fixed_point_free_elements_in_Hol(Zpv)_p_odd}. Before formulating the lemma, we remark that we will slightly abuse language: namely, we call an element of $i \in \Z_{2n}$ \emph{even} resp. \emph{odd} if it is represented in $\Z$ by an even resp. odd integer. Observe that this is independent of the choice of a representative!

\begin{lem} \label{lem:fixed_point_free_elements_in_Hol(Z2v)}
    Let $v > 0$ and let $g \in \Hol(\Z_{2^v})$ be such that $g^2 = 1$ and $f(x) \neq x$ holds for all $x$, then one of the following cases holds:
    \begin{enumerate}
        \item $g(x) = x + 2^{v-1}$.
        \item There is an odd $\beta \in \Z_{2^v}$ such that $g(x) = \beta - x$.
    \end{enumerate}
\end{lem}

\begin{proof}
    For $v = 1$, we can easily check the validity of the lemma by hand. Now, let $v > 1$ and write $g(x) = \alpha x + \beta$ with $\alpha, \beta \in \Z_{2^v}$ where $\alpha$ is invertible. As in the proof of \cref{lem:fixed_point_free_elements_in_Hol(Zpv)_p_odd}, we argue that $\alpha^2 = 1$ which implies that $\alpha \in \{ \pm 1, 2^{v-1}\pm 1 \}$ (see \cite[Lemma 2.1.21]{cohen_nt}).

    If $\alpha = 1$, then $g(x) = x + \beta$ and $g^2(x) = x + 2 \beta = x$, which, considering that $g$ is fixed-point-free, implies $\beta = 2^{v-1}$. On the other hand, if $\alpha = -1$, then $g(x) = \beta - x$, which satisfies $g^2(x) = x$ unconditionally. It is easily seen that $x$ is a fixed point of $g$ if $2x = \beta$ which shows that $g$ is fixed-point-free if and only if $\beta$ is odd.

    Note that for $v = 2$, the cases $\alpha = \pm 1$ already cover all possible cases, so we will now assume that $v > 2$.
    
    If $\alpha = 2^{v-1}+ 1$, then $g^2(x) = x + (2^{v-1} + 2)\beta = x$. As $v > 2$, we see that $2^{v-1} + 2$ is divisible by $2$ but not by $4$ which implies $\beta = i \cdot 2^{v-1}$ for some integer $i$. But if $g(x) = (2^{v-1}+1)x + i \cdot 2^{v-1}$, we get
    \[
    g(-i) = (2^{v-1}+1)(-i) + i \cdot 2^{v-1} = -i,
    \]
    i.e. $g$ is not fixed-point-free. On the other hand, if $g(x) = (2^{v-1} - 1)x + \beta$, then $g^2(x) = x + 2^{v-1}\beta = x$, which implies that $\beta = 2i$ for some integer $i$. As $2^{v-1}-2$ is divisible by $2$ but not by $4$, there is a $j \in \Z_{2^v}$ such that $(2^{v-1}-2)j = 2$. Now
    \[
    g(-ij) = (2^{v-1}-1) \cdot (-ij) + 2i = -ij - (2^{v-1}-2) \cdot ij + 2i = -ij -2i + 2i = -ij, 
    \]
    again showing that $g$ has fixed points. As a consequence, $\alpha = \pm 1$, thus proving the lemma.
\end{proof}

We can now start with the proof our second main result. The proof of \cref{thm:full_diagonal_implies_retractable_2_type} will be by an inductive argument. In order to increase readability of the proof, we decided to outsource several arguments in the induction step by stating them as separate lemmata.

\begin{lem} \label{lem:fix_of_2_type}
    Suppose that $v \geq 1$ and that $X = \{ x_i : i \in \Z_{2^v} \}$ is an irretractable cycle set with diagonal $T(x_i) = x_{i+1}$. Let $f \in \Fix(\G(X))$ with $o_{\circ}(f) = 2$, then $\lambda_f(x_i) = x_{i + 2^{v-1}}$ for $x_i \in X$.
\end{lem}

\begin{proof}
    By \cref{pro:fix_provides_automorphisms}, $\lambda_f$ centralizes $T$, which implies that $\lambda_f$ centralizes the regular action of the abelian group $\gen{T} \cong \Z_{2^v}$ on $X$. As this implies $\lambda_f \in \gen{T}$, the condition that $o_{\circ}(f) = 2$ now implies that $\lambda_f(x_i) = x_{i + 2^{v-1}}$. 
\end{proof}

We now show that the center of a potential counterexample to the theorem can be described very precisely:

\begin{lem} \label{lem:center_of_counterexample}
    Suppose that $v \geq 2$ and that $X = \{ x_i : i \in \Z_{2^v} \}$ is an irretractable cycle set with diagonal $T(x_i) = x_{i+1}$. If $z \in \Zen(\G(X))$ with $o_{\circ}(z) = 2$ then $\lambda_z(x_i) = x_{\beta-i}$ for $x_i \in X$, where $\beta$ is odd. Furthermore, $z$ is the only element of order $2$ in $\Zen(\G(X))$.
\end{lem}

\begin{proof}
    Suppose that $X$ is as stated. By \cref{lem:fix_of_2_type}, we find an element $f \in \Fix(\G(X))$ such that $\lambda_f(x_i) = x_{i + 2^{v-1}}$.

    Now let $z \in \Zen(\G(X))$ be such that $o_{\circ}(z) = 2$. By \cref{thm:diagonal_commutes_with_conjugated_diagonal}, we argue that $T^z \in C_{\Sym_{2^v}}(T) = \gen{T}$ and thus, $\lambda_z \in N_{\Sym_{2^v}}(\gen{T})$. As $\lambda_z$ acts without fixed points and normalizes the regular action of $\gen{T} \cong \Z_{2^v}$ on $X$, \cref{lem:fixed_point_free_elements_in_Hol(Z2v)} implies that one of the following cases must occur:
    \begin{enumerate}
        \item $\lambda_{z}(x_i) = x_{i+2^{v-1}}$, or
        \item $\lambda_{z}(x_i) = x_{\beta - i}$ for some odd integer $\beta$.
    \end{enumerate}
    In the first case, we see that $\lambda_z = \lambda_f$, which by \cref{pro:center_and_fix_intersect_in_socle} implies $\Soc(\G(X)) \neq 0$. But then $X$ would be retractable by \cref{thm:socle_detects_retractability}. Therefore, only the second case can occur. This proves the first part of the proposition.

    Suppose now that there was a $z' \in \Zen(\G(X))$ with $o_{\circ}(z') =2$, then $\lambda_{z'}(x_i) = x_{\beta'-i}$ with some odd $\beta' \neq \beta$. But then $z'\circ z \in \Zen(\G(X))$ and $o_{\circ}(z'\circ z)=2$ while  $\lambda_{z' \circ z}(x_i) = x_{i+\beta-\beta'}$, contradicting the first statement of the proposition.
\end{proof}

We now prove a series of statements that will be used in the induction step for the proof of \cref{thm:full_diagonal_implies_retractable_2_type}.

\begin{lem} \label{lem:properties_of_and_reduction_to_xphi}
    Let $v \geq 5$ and suppose that $X$ be an irretractable cycle set of size $n = 2^v$ such that $T$ is a full $n$-cycle. Let $z \in \Zen(\G(X))$ be such that $o_{\circ}(z) = 2$ and consider $\varphi = \id - \lambda_z \in \Endl(\G(X))$. Then:
    \begin{enumerate}[label=\alph*)]
        \item $T^z = T^{-1}$.
        \item $T_{\varphi} = T_{X_{\varphi}} = T^2$
        \item $2^{v-1}$ divides $d(X_{\varphi})$.
        \item $d(X_{\varphi})$ divides $o_+(z)$. In particular, $k \geq v-1$, where $o_+(z) = 2^k$.
        \item $X_{\varphi}$ is irretractable.
        \item Let $h = 2^{k-1}z$, then $h \in \Fix(\G(X_{\varphi}))$
    \end{enumerate}
\end{lem}

\begin{proof}
    Identifying $X = \{x_i : i \in \Z_{2^v} \}$ and $T(x_i) = x_{i+1}$, we apply \cref{lem:center_of_counterexample} in order to describe $\lambda_z(x_i) = x_{\beta -i}$ for some odd $\beta$. In order to prove a), we calculate
    \[
    T^z(x_i) = \lambda_z(T(x_{\beta-i})) = \lambda_z(x_{\beta-i+1}) = x_{\beta-(\beta-i+1)} = x_{i-1} = T^{-1}(x_i).
    \]
    Applying \cref{pro:diagonals_are_homomorphisms} and \cref{pro:diagonal_of_cabling_by_center}, we derive b) by calculating
    \[
        T_{\varphi} = T_{\id} \circ T_{\lambda_z}^{-1} = T \circ (T^{z})^{-1} = T \circ (T^{-1})^{-1} = T^2.
    \]
    By \cref{thm:oT_divides_dX}, it now follows that $d(X_{\varphi})$ is divisible by $o(T_{\varphi}) = o(T^2) = 2^{v-1}$, therefore proving c).

    In order to prove d), note that, by \cref{pro:properties_of_phi_z}d), $\varphi(g) = z - \lambda_g(z)$ for $g \in \G(X)$, so $o_+(\varphi(g))$ divides $o_+(z)$ for all $g \in \G(X)$. As $\G(X_{\varphi}) = \varphi(\G(X))$ by \cref{pro:permutation_group_of_endocabled_cycle_set}, it follows that $d(X_{\varphi}) = \exp_+(\varphi(\G(X)))$ divides $o_+(z)$. As a consequence of c), $o_+(z) = 2^k$ is divisible by $2^{v-1}$, therefore $k \geq v-1$.

    In order to prove e), suppose that $X_{\varphi}$ is retractable. By c), $d(X_{\varphi})$ is divisible by $4$, as $v-1 \geq 4$. Note that $\G(X)$ is a $2$-group by \cref{cor:full_n_cycle_implies_pi_type}, so by \cref{pro:replacement_of_center}, we find a $z' \in \Zen(\G(X))$ with $o_{\circ}(z') = 2$ and $z \neq z'$. As this contradicts the second statement in \cref{lem:center_of_counterexample}, we conclude that $X_{\varphi}$ is irretractable.

    For the proof of e), let $h = 2^{k-1}z$. From \cref{pro:properties_of_phi_z}c), it follows for $g \in \G(X_{\varphi})$ that
    \[
    \lambda_g(h) = 2^{k-1} \lambda_g(z) = 2^{k-1}(z-2g) = h - 2^kg = h
    \]
    where $2^kg = 0$ follows from d) which implies that $\exp_+(\G(X_{\varphi}))$ divides $o_+(z)$. Therefore, $h \in \Fix(\G(X_{\varphi}))$.
\end{proof}

We now want to understand the subbrace generated by the center of a potential counterexample, together with its action on $X$. For one of the statements, we will need the following lemma:

\begin{lem}\label{lem:centralizer_of_T2}
    Let $v \geq 2$ and let $\pi \in \Sym_{\Z_{2^v}}$ be defined by $\pi(i) = i+2$. Let $\rho \in C_{\Sym_{\Z_{2^v}}}(\pi)$ be such that $\rho^2 = \id_{\Z_{2^v}}$ and $\rho(i) \neq i$ for all $i \in \Z_{2^v}$. Then exactly one of the following is true:
    \begin{enumerate}
        \item $\rho(i) = i + 2^{v-1}$ for all $i \in \Z_{2^v}$.
        \item There is an odd $\gamma \in \Z_{2^v}$ such that
        \[
        \rho(i) = \begin{cases}
            i + \gamma & i \textnormal{ even}, \\
            i - \gamma & i \textnormal{ odd}.
        \end{cases}
        \]
    \end{enumerate}
\end{lem}

\begin{proof}
    Let $\Orb_e, \Orb_o$ be the orbits of $\pi$ on $\Z_{2^v}$ which consist of the even and the odd elements, respectively. $\rho$ either fixes or switches the orbits of $\pi$ setwise. In the first case, that is $\rho(\Orb_e) = \Orb_e$, $\rho(\Orb_o) = \Orb_o$, we see that $\rho$ centralizes the regular action of the abelian group $\gen{\pi}$ on the respective orbits, showing that there are $\gamma_e,\gamma_o \in \Z_{2^v}$ such that
    \[
    \rho(i) = \begin{cases}
        i + \gamma_e & i \in \Orb_e, \\
        i + \gamma_o & i \in \Orb_o.
    \end{cases}
    \]
    As $\rho^2 = \id_{\Z_{2^v}}$ and $\rho$ is fixed-point free, we infer that $\gamma_e = \gamma_o = 2^{v-1}$, therefore showing that $\rho$ is as in case (1).
    
    If, however, $\rho(\Orb_e) = \Orb_o$ and $\rho(\Orb_o) = \Orb_e$, let $\rho(0) = \gamma$. Clearly, $\gamma$ is odd. For all $k$, we now calculate
    \[
    \rho(2k) = \rho(\pi^k(0)) = \pi^k(\rho(0)) = \pi^k(\gamma) = \gamma + 2k,
    \]
    showing that $\rho(i) = i+\gamma$ for $i \in \Orb_e$. As $\rho^2 = \id_{\Z_{2^v}}$, it follows that $\rho(i) = i-\gamma$ for $i \in \Orb_o$, thus showing that $\rho$ is as in case (2).
\end{proof}

\begin{lem} \label{lem:ingredients_for_induction_step}
    Suppose that \cref{thm:full_diagonal_implies_retractable_2_type} is true for all cycle sets of size $2^{v-1}$. Let $X$ be an irretractable cycle set of size $n = 2^v$ ($v \geq 5$) such that $T$ is a full $n$-cycle. Let $z \in \Zen(\G(X))$ be such that $o_{\circ}(z) = 2$ and let $\varphi = \id - \lambda_z \in \Endl(\G(X))$. Furthermore, let $o_+(z) = 2^k$ in the brace $\G(X)$, then
    \begin{enumerate}[label=\alph*)]
        \item $X_{\varphi}$ is irreducible and, in particular, indecomposable.
        \item With $h = 2^{k-1}$, the permutation $\lambda_h$ does not have any fixed points on $X$.
        \item $k \in \{v-1,v \}$.
        \item The subbrace $B = \gen{z}_+$ acts semiregularly on $X$.
    \end{enumerate}
\end{lem}

\begin{proof}
    For the proof of a), suppose now that there is a proper sub-cycle set $\emptyset \neq Y \subsetneq X_{\varphi}$. As $T_{\varphi} = T^2$, the same argument as in the proof of \cref{pro:full_diagonal_implies_no_sub_cyclesets} shows that $Y$ coincides with one of the two cycles of $T^2$, say $Y = \{ x_{2i} : i \in \Z_{2^v} \}$. But then, $T_Y$ is a full $2^{v-1}$-cycle on $Y$. If the statement of \cref{thm:full_diagonal_implies_retractable_2_type} is assumed to be true for cycle sets of size $2^{v-1}$, then $Y$ is retractable, that is, there are $x,y \in Y$ such that $x \cab{\varphi} x_{2i} = y \cab{\varphi} x_{2i}$ for all $i \in \Z_{2^v}$. But $\lambda_z$ centralizes $\G(X_{\varphi})$, so we also see that
    \begin{align*}
        x \cab{\varphi} x_{2i+1} & = \lambda_z^{-1}(x \cab{\varphi} (\lambda_z(x_{2i+1}))) \\ & = \lambda_z^{-1} (x \cab{\varphi} \underbrace{x_{\beta -(2i+1)}}_{\in Y})  \\ & = \lambda_z^{-1}(y \cab{\varphi} x_{\beta - (2i+1)})  \\ & = \lambda_z^{-1}(y \cab{\varphi} (\lambda_z(x_{2i+1})))  \\ & = y \cab{\varphi} x_{2i+1}.
    \end{align*}
    This implies that $x \cab{\varphi} z = y \cab{\varphi} z$ for \emph{all} $z \in X$, with the consequence that $X_{\varphi}$ is retractable, contradicting \cref{lem:properties_of_and_reduction_to_xphi}e). Therefore, $X_{\varphi}$ is irreducible and, in particular, indecomposable.

    It follows from \cref{lem:properties_of_and_reduction_to_xphi}f) together with \cref{pro:fix_provides_automorphisms} that $\lambda_h$ is an automorphism of $X_{\varphi}$. Let $F = \{x \in X: \lambda_h(x) = x \}$, then $F$ is a sub-cycle set of $X_{\varphi}$. Under the assumption that \cref{thm:full_diagonal_implies_retractable_2_type} is true for cycle sets of size $2^{v-1}$, either $F = \emptyset$ or $F = X$ due to part a). As $h \neq 0$, it follows that $F = \emptyset$, thus proving b).

    By \cref{lem:properties_of_and_reduction_to_xphi}d), we know that $k \geq v-1$. In order to prove c), we show $k \leq v$. In order to achieve that, let $B = \gen{z}_+$. As $v \geq 5$, we know from \cref{pro:central_involutions_generate_Bk} that $B \cong B_k$. Now pick an $\Tilde{x} \in X$ with $\lambda_h(\Tilde{x}) \neq \lambda_z(\Tilde{x})$ and consider $\Orb = \Orb_B(\Tilde{x})$, its orbit under $B$.
    By \cref{pro:structure_of_bk}, the only elements of order $2$ in $B$ are $h,z$ and $h \circ z$. From b) and our assumption on $x$ it follows that no element of order $2$ in $B$ fixes $x$. As $B$ is a $2$-group, this implies that $B$ acts regularly on $\Orb$, therefore proving $|B| \leq |X|$, which implies $k \leq v$.

    The argument in the proof of part c) already proves statement d) in the case when $k = v$, so suppose now that $k = v-1$ and that $B$ does \emph{not} act semiregularly on $X$. Let $f \in \Fix(\G(X))$ be such that $o_{\circ}(f) = 2$, then by \cref{lem:fix_of_2_type}, $f$ acts on $X$ by $\lambda_f(x_{i+2^{v-1}})$, so $\lambda_f(x) \neq \lambda_z(x)$ for all $x \in X$. If we had $h = f$, the method of proof in part c) would show that $B$ \emph{does} act semiregularly on $X$. Therefore $h \neq f$. Using \cref{pro:fix_of_brace_provides_automorphisms}, we calculate for all $mz \in B$ that
    \[
    {}^{\circ f}(mz) = \lambda_f(mz) = m \cdot \lambda_f(z) = m \cdot {}^{\circ f}z = mz,
    \]
    implying that $f$ centralizes $B$. As before, let $\Tilde{x} \in X$ be such that $B$ acts regularly on $\Orb = \Orb_B(x)$. We must have $\lambda_f(\Orb) = \Orb$, else the abelian group $\gen{B \cup \{ f \}}_{\circ}$ would act transitively on $X$, thus implying semiregularity of the action of $B$. As $\lambda_f(\Orb) = \Orb$, $f$ has to centralize the regular action of the abelian group $B^{\circ}$ on $\Orb$, thus there is a $g \in B$ such that $\lambda_g(x) = \lambda_f(x)$ for all $x \in X$. As $o_{\circ}(f) = 2$, $g \in \{ h,z \circ h, z \}$. Clearly, $g \neq z$ as $\lambda_f(x) = \lambda_z(x)$ for all $x \in X$. As $h$ centralizes $T_{\varphi} = T^2$ and $h \neq f$, we use \cref{lem:centralizer_of_T2} to argue that there is an odd $\gamma \in \Z_{2^v}$ such that $\lambda_h(x_i) = x_{i \pm \gamma} \neq x_{i + 2^{v-1}} = \lambda_f(x_i)$ for all $x_i \in X$. Therefore, $g \neq f$. Finally, suppose that $g = z \circ h$. As $\lambda_h$ is an involution on $\Orb$, $\lambda_h(x_i) = x_{i+\gamma}$ holds for exactly $\frac{|\Orb|}{2} = 2^{v-2}$ elements in $\Orb$. Let $x_i \in \Orb$ be such that $\lambda_h(x_i) = x_{i+ \gamma}$, then
    \[
    \lambda_f(x_i) = \lambda_{z \circ h}(x_i) \Rightarrow i + 2^{v-1} = 1-\gamma - i \Rightarrow i \in \left\{ \frac{1-\gamma}{2} + 2^{v-2}, \frac{1-\gamma}{2} + 2^{v-2} + 2^{v-1} \right\}.
    \]
    Therefore, there are only $2 < 2^{v-2}$ such elements, a contradiction! Therefore, $B$ acts semiregularly on $X$.
\end{proof}

We now use the technique of \emph{cabling out the center} in order to exclude $k = v$ in the above proposition:

\begin{lem} \label{lem:counterexample_has_k=v-1}
    Suppose that \cref{thm:full_diagonal_implies_retractable_2_type} is true for all cycle sets of size $2^{v-1}$. Let $X$ be an irretractable cycle set of size $n = 2^v$ ($v \geq 5$) such that $T$ is a full $n$-cycle. Let $z \in \Zen(\G(X))$ be such that $o_{\circ}(z) = 2$, then $o_+(z) = 2^{v-1}$.
\end{lem}

\begin{proof}
    Suppose that $o_+(z) = 2^v$. Let $\varphi = \id - \lambda_z \in \Endl(\G(X))$. Then by \cref{pro:cabling_out_the_center} and \cref{pro:permutation_group_of_endocabled_cycle_set}, $z \not\in \varphi(\G(X)) = \G(X_{\varphi})$. As $\G(X_{\varphi}) \leq \G(X)$ is a $2$-group, there is a $z' \in \G(X_{\varphi})$ with $o_{\circ}(z')= 2$. Clearly, $z' \neq z$. Let $B= \gen{z}_+$. As $z'$ centralizes $z$ and $\gen{2z}_+ = \varphi(\gen{z}_+) \subseteq \G(X_{\varphi})$, we see that $z'$ centralizes $B$ by \cref{pro:structure_of_bk}a). Now by \cref{lem:ingredients_for_induction_step}d), $B$ acts regularly on $X$. As $B$ is abelian, $z' \in B$. In particular, $z' \in B \cap \G(X_{\varphi})$. As $z \not\in \G(X_{\varphi})$ and $2z \in \G(X_{\varphi})$, we see that $B \cap \G(X_{\varphi}) = \gen{2z}_+$. By \cref{pro:structure_of_bk}b), $h = 2^{k-1}z$ is the only element of order $2$ in $\gen{2z}_+$, therefore
    \[
    z' = h \in \Zen(\G(X_{\varphi})) \cap \Fix(\G(X_{\varphi})) \leq \Soc(\G(X_{\phi}))
    \]
    by \cref{pro:center_and_fix_intersect_in_socle}, so $\G(X_{\varphi})$ is retractable by \cref{thm:socle_detects_retractability}, contradicting \cref{lem:properties_of_and_reduction_to_xphi}e). It follows that $o_+(z) \neq 2^v$, so by \cref{lem:ingredients_for_induction_step}c) it follows that $o_+(z) = 2^{v-1}$.
\end{proof}

We are now in the position of finishing the proof of \cref{thm:full_diagonal_implies_retractable_2_type}. We first set up the starting stone for our inductive argument:

\begin{lem} \label{lem:start_the_induction}
    The statement of \cref{thm:full_diagonal_implies_retractable_2_type} is true if $v \in \{ 3,4 \}$.
\end{lem}

\begin{proof}
    For $v = 3$, the statement is easily confirmed by checking all cycle sets of size $8$ which are part of the GAP library \texttt{YangBaxter}. For $v=4$, \cref{lem:center_of_counterexample} shows that if there is a counterexample, then there is one with $X = \{ x_i : i \in \Z_{16} \}$, the diagonal $T(x_i) = x_{i+1}$ and $\lambda_z(x_i) = x_{\beta - i}$ for an element $z \in \Zen(\G(X))$ and an odd $\beta$. Note that $\lambda_z(x_{i+\delta}) = x_{(\beta - 2\delta - i) +\delta}$ and $T(x_{i+\delta}) = T(x_{(i+1) + \delta})$, so by a suitable relabelling, we can assume that a counterexample has the diagonal $T(x_i)= x_{i+1}$, for which there is a $z \in \Zen(\G(X))$ with $\lambda_z(x_{1-i})$.

    Taking into account \cref{lem:fix_of_2_type}, we can furthermore assume that $x_i \mapsto x_{i+8}$ is an automorphism of $X$. We therefore model $X$ as a map $C: \Z_{16} \times \Z_{16} \to \Z_{16}$ that describes the operation on $X$ by $x_i \ast x_j = x_{C(i,j)}$. The map $C$ is characterized by the following properties:
    \begin{enumerate}
        \item $\forall i,j,k \in \Z_{16}: (j \neq k) \Rightarrow (C(i,j) \neq C(i,k))$,
        \item $\forall i,j,k \in \Z_{16}: C(C(i,j),C(i,k)) = C(C(j,i),C(j,k))$,
        \item $\forall i \in \Z_{16}: C(i,i) = i+1$,
        \item $\forall i,j \in \Z_{16}: C(i,1-j) = 1 - C(i,j)$,
        \item $\forall i,j \in \Z_{16}: C(i+8,j+8) = C(i,j) + 8$,
        \item $\forall i,j \in \Z_{16}: (i \neq j) \Rightarrow (\exists k \in \Z_{16}: C(i,k) \neq C(j,k))$.
    \end{enumerate}
    Note that the last condition reflects the irretractability of $X$. We used Savile Row 1.10, a modelling assistent for constraint programming \cite{savilerow} to check for the nonexistence of such a cycle set. The written code is provided in the appendix.
\end{proof}

\begin{proof}[Proof of \cref{thm:full_diagonal_implies_retractable_2_type}]

    By \cref{lem:start_the_induction}, the theorem is true for $v = \{3,4 \}$, so suppose that $v \geq 5$ and that the theorem holds for all cycle sets of size $2^{v-1}$. Without restriction, we may assume that the cycle set in question is defined on the set $ X = \{x_i : i \in \Z_{2^v} \}$ with $T(x_i) = x_{i+1}$. By \cref{cor:full_n_cycle_implies_pi_type}, $\G(X)$ is a $2$-group, so there is a $z \in \Zen(\G(X))$ with $o_{\circ}(z) = 2$.
    
    We want to prove that $X$ is retractable. We can already reduce this problem to the case when:
    \begin{enumerate}
        \item There is a unique $z \in \Zen(\G(X))$ with $o_{\circ}(z) = 2$. This element acts on $X$ as $\lambda_z(x_i) = x_{\beta-i}$ for some odd $\beta \in \Z_{2^v}$ (\cref{lem:center_of_counterexample}).
        \item $o_+(z) = 2^{v-1}$ (\cref{lem:counterexample_has_k=v-1}).
        \item For $\varphi = \id - \lambda_z$, the cycle set $X_{\varphi}$ is irretractable (\cref{lem:properties_of_and_reduction_to_xphi}e)) and indecomposable (\cref{lem:ingredients_for_induction_step}a)). Furthermore, $d(X_{\varphi}) =2^{v-1}$ (\cref{lem:properties_of_and_reduction_to_xphi}c)d)),
    \end{enumerate}
    as all other cases have already been dealth with. We can assume without restriction that all of these conditions are satisfied.

    As in the proof of \cref{lem:counterexample_has_k=v-1}, let $z' \in \Zen(\G(X_{\varphi}))$ be such that $o_{\circ}(z) = 2$. By \cref{pro:cabling_out_the_center}, $z \neq z'$. As $X_{\varphi}$ is indecomposable, $\lambda_{z'}$ acts without fixed points (\cref{pro:centralizer_of_transitive_permutation_group_is_semiregular}).

    Let $B = \gen{z}_+$. We claim that $z'$ can be chosen in a way such that $B' = \gen{B \cup \{ z' \}}_{\circ}$ acts transitively on $X$. Suppose first that $B' = \gen{B \cup \{ z' \}}_{\circ}$ does \emph{not} act transitively on $X$ and, recalling that $B$ acts semiregularly on $X$ (\cref{lem:ingredients_for_induction_step}d)), let $X_1,X_2 \subseteq X$ be the orbits of $B$ on $X$. As $B'$ is not transitive on $X$, it is necessary that $\lambda_{z'}(X_i) = X_i$ ($i=1,2$). We conclude that there are $g_1,g_2 \in B$ with $o_{\circ}(g_i)=2$ ($i=1,2$) such that
    \[
    \lambda_{z'} (x) = \begin{cases}
        \lambda_{g_1}(x) & x \in X_1 \\
        \lambda_{g_2}(x) & x \in X_2.
    \end{cases}
    \]
    If $g_1 = g_2$, then actually, $z' = g_1 = g_2 \in B$, and we can conclude as in the proof of \cref{lem:counterexample_has_k=v-1} that $z' \in \Soc(\G(X_{\varphi}))$, thus proving that $X_{\varphi}$ is retractable, implying retractability of $X$ via \cref{lem:properties_of_and_reduction_to_xphi}e).
    
    Otherwise as $\lambda_z(x) \neq \lambda_{z'}(x)$ for all $x \in X$ (\cref{pro:centralizer_of_transitive_permutation_group_is_semiregular}), it is only possible (relabelling, if necessary), that $g_1 = h$, $g_2 = z \circ h$, where $h = 2^{v-2}z$. We can now determine $X_1,X_2$: using \cref{thm:diagonal_commutes_with_conjugated_diagonal} and \cref{lem:properties_of_and_reduction_to_xphi}b), we infer
    \[
        T_{\varphi}^{z'} \circ T_{\varphi} = T_{\varphi} \circ T_{\varphi}^{z'}
        \Rightarrow T^2 \circ (T^2)^{z'} = (T^2)^{z'} \circ T^2.
    \]
    As $(T^2)^{z'}$ centralizes $T^2$, $(T^2)^{z'}$ fixes or exchanges the $2$ orbits of $\gen{T^2}$. Therefore, $((T^2)^{z'})^2 = (T^4)^{z'}$ fixes the orbits of $\gen{T^2}$ and therefore centralizes the respective actions of $\gen{T^2}$ on $X_1$ and $X_2$. Together with the observation that $(T^4)^{z'}$ is a product of $4$ cycles of size $2^{v-2}$, it follows that the cycles of $(T^4)^{z'}$ are, as sets, identical to the cycles of $T^4$, which are given by
    \[
    C_j = \{ x_l : l \in j + 4 \cdot \Z_{2^v}  \} \quad (0 \leq j \leq 3).
    \]
    In particular, $\lambda_{z'}$ permutes these sets. Recall that $\lambda_{z}(x_i) = x_{\beta - i}$ with odd $\beta$, so $\lambda_z(C_j) \neq C_j$ for all $j$. On the other hand, $\lambda_h$ centralizes $T^4$ by \cref{pro:fix_provides_automorphisms}, so $\lambda_h(C_j) = C_j$ for all $j$. From this, we see that both $X_1$ and $X_2$ are unions of some of the $C_j$.

    In particular, it follows that $X_2$ is stable under $T^4$. As $X_2$ is also stable under $\lambda_{z'} = \lambda_{z\circ h}$, it follows that for $x \in X_2$,
    \begin{equation} \label{eq:T4z'}
        (T^4)^{z'}(x)  = ((T^2)^{h\circ z})^2(x)
     = ((T^2)^z)^2(x)
     = (T^{-2})^2(x) = T^{-4}(x),
    \end{equation}
    where the second equality follows from the fact that $\lambda_h$ centralizes $T^2 = T_{\varphi}$ due to $h \in \Fix(\G(X_{\varphi}))$ (\cref{pro:fix_provides_automorphisms}). Putting $\varphi' = \id - \lambda_{z'}$ and $\varphi'\circ \varphi = \Tilde{\varphi}$, \cref{pro:adding_two_cablings} and \cref{pro:diagonal_of_cabling_by_center} imply
    \[
    T_{\Tilde{\varphi}}  := T_{(X_{\varphi})_{\varphi'}} 
         = T_{\varphi} \circ (T_{\varphi}^{z'})^{-1} 
         = T^2 \circ (T^{-2})^{z'}.
    \]
    Using \cref{eq:T4z'}, we deduce for $x \in X_2$ that
    \[
    T_{\Tilde{\varphi}}^2(x) = (T^4 \circ (T^{-4})^{z'})(x) = T^4 (T^4(x)) = T^8(x).
    \]
    As $v \geq 5$, it follows that $o(T_{\Tilde{\varphi}}^2) = 2^{v-3}$ is divisible by $2$ which implies that $o(T_{\Tilde{\varphi}}) = 2^{v-2}$ which is divisible by $4$. In particular, we note that $d((X_{\varphi})_{\varphi'})$ is divisible by $4$.

    By \cref{pro:centralizer_of_transitive_permutation_group_is_semiregular}, $\lambda_{z'\circ z}(x) \neq x$ for any $x\in X$. Using \cref{pro:properties_of_phi_z}b) for $z$ and $z'$, we calculate for any $y \in X$:
    \begin{align*}
    \lambda_{z' \circ z}(x) \cab{\Tilde{\varphi}} y  & = \lambda^{-1}_{\Tilde{\varphi}(\lambda_{\lambda_{z' \circ z}(x)})}(y) \\ &= \lambda^{-1}_{\Tilde{\varphi}(\lambda_{z' \circ z}(\lambda_x))}(y)\\ &= \lambda^{-1}_{\lambda_{z' \circ z}(\Tilde{\varphi}(\lambda_x))}(y)\\ &= \lambda^{-1}_{\lambda_{z'}(-\Tilde{\varphi}(\lambda_x))}(y)\\ &= \lambda^{-1}_{\Tilde{\varphi}(\lambda_x)}(y)\\ &= x \cab{\Tilde{\varphi}} y.
    \end{align*}
    It follows that $(X_{\varphi})_{\varphi'}$ is retractable and, as $d(X_{\varphi})_{\varphi'}$ is divisible by $4$, we can find $z'' \in \Zen(\G(X_{\varphi}))$ with $o_{\circ}(z'') = 2$ and $z'' \neq z'$. If $\gen{B \cup \{ z'' \}}_{\circ}$ still does not act transitively on $X$, then we can argue similarly as above that
    \[
    \lambda_{z''} (x) = \begin{cases}
        \lambda_{z \circ h}(x) & x \in X_1 \\
        \lambda_{h}(x) & x \in X_2.
    \end{cases}
    \]
    However, this is not possible, as then, $z' \circ z'' \in \G(X_{\varphi})$ and $\lambda_{z'\circ z''}(x) = \lambda_z(x)$ for all $x \in X$, despite $z \not\in \G(X_{\varphi})$.

    We can therefore choose $z'$ in a way that $B' = \gen{B \cup \{ z' \}}_{\circ}$ acts transitively on $X$ which is from now on our assumption.

    We claim that $B'$ is, in fact a subbrace of $\G(X)$. First observe that $\Tilde{B} = B + \gen{z'}_+$ is an abelian subbrace of $B$: it follows from \cref{pro:central_involutions_generate_Bk} that $\gen{z'}^{\circ}_+$ is abelian. Furthermore, $z'$ centralizes $B$ and $z$ centralizes $\gen{z'}_+$, so $\gen{z'}_+$ and $B$ centralize one another by \cref{pro:commutativity_of_cyclic_subbraces}. Furthermore, $\Tilde{B}$ is closed under multiplication as
    \[
    az' \circ bz = az' + b\lambda_{az'}(z) = az' + b(z - 2az') \in \Tilde{B}
    \]
    for all $az',bz \in \Tilde{B}$, so $\Tilde{B}$ is an abelian subbrace that contains $B'$, which acts transitively on $X$, so $\Tilde{B} = B'$ because $|B'| > |B| = 2^{v-1}$. It follows that $B'$ is a subbrace of $\G(X)$.
    
    As $(B:B')=2$, it now follows that $2z' \in B = \gen{z}_+$. Thus, we can write $2z' = mz$. Now let $g \in \G(X_{\varphi})$. We calculate
    \begin{align*}
        \lambda_{z'}(2g) & = \lambda_{z'}(\varphi(g)) \qquad (\textnormal{\cref{pro:properties_of_phi_z}a)}) \\ 
        & = \lambda_{z'}(z - \lambda_g(z)) \qquad (\textnormal{\cref{pro:properties_of_phi_z}d)})) \\
        & = \lambda_{z'}(z) - \lambda_g(\lambda_{z'}(z)) \\
        & = (z - 2z') - \lambda_g(z-2z') \qquad (\textnormal{\cref{pro:properties_of_phi_z}c)}) \\
        & = (1-m)z - \lambda_g((1-m)z) \\
        & = (1-m)  (z-\lambda_g(z)) \\
        & = (1-m)  2g.
    \end{align*}
    Putting $m' = 1-m$, it follows that $\lambda_{z'}(g) = m'g$ for $g \in 2\G(X_{\varphi}) = \G(X_{2\varphi})$. As $d(X_{\varphi}) = 2^{v-1}$, we see that $d(X_{2\varphi}) = 2^{v-2}$. As $\lambda_{z'}(g) = m'g$ is an involution on $\G(X_{2\varphi})$, which is of exponent $2^{v-2}$, we conclude that $m'^2 = 1$ in $\Z_{2^{v-2}}$ and therefore, $m' \in \{ \pm 1, 2^{v-3} \pm 1 \}$ (\cite[Lemma 2.1.21]{cohen_nt}).

    Now let $g \in \G(X_{4\varphi}) = 2 \G(X_{2\varphi})$ and write $g = 2g'$ with $g' \in \G(X_{2\varphi})$. Assume furthermore that $m' = 2^{v-3} \pm 1$ (with one choice of sign), then
    \[
    \lambda_{z'}(g) = (2^{v-3} \pm 1) \cdot 2g' = 2^{v-2}g' \pm 2g' = \pm 2g' = \pm g.
    \]
    Therefore, $\lambda_{z'}$ acts on $\G(X_{4\varphi})$ as $\id$ or $-\id$. If $\lambda_{z'}(g) = g$ for all $g \in \G(X_{4\varphi})$, then for any $x,y \in X$, we have
    \[
        \lambda_{z'}(x) \cab{4\varphi} y  = \lambda^{-1}_{4\varphi(\lambda_{\lambda_{z'}(x)})}(y) 
         = \lambda^{-1}_{4\varphi(\lambda_{z'}(\lambda_x))}(y) 
         = \lambda^{-1}_{\lambda_{z'}(4\varphi(\lambda_x))}(y) 
         = \lambda^{-1}_{4\varphi(\lambda_x)}(y) 
         = x \cab{4\varphi} y.
    \]
    As $x \neq \lambda_{z'}(x)$, it follows that $X_{4\varphi}$ is retractable. Furthermore, as $v \geq 5$, the class $d(X_{4\varphi}) = \exp_+(4\G(X_{\varphi})) = 2^{v-3}$ is divisible by $4$, so \cref{pro:replacement_of_center} implies the existence of an element $z'' \in \Zen(\G(X))$ with $o_{\circ}(z'')=2$ and $z''\neq z$, in contradiction to \cref{lem:center_of_counterexample}.

    On the other hand, if $\lambda_{z'}(g) = -g$ for $g \in \G(X_{4\varphi})$, then for $g \in \G(X_{4\varphi})$ an application of \cref{pro:properties_of_phi_z}a) shows $\lambda_{z'\circ z}(g) = \lambda_{z'} (-g) = g$. Now the same argument as above shows $\lambda_{z'\circ z}(x) \cab{4\varphi} y = x \cab{4\varphi} y$ for all $x,y \in X$. As $\lambda_{z'\circ z}(x) \neq x$ for $x \in X$ (\cref{pro:centralizer_of_transitive_permutation_group_is_semiregular}), it follows again that $X_{4\varphi}$ is retractable, providing the same contradiction as before.

    We have therefore shown that $X$ can not be irretractable, thus proving the theorem.
\end{proof}

\section{Some problems}

We close this work by posing some open problems.

\begin{problem}
    Develop methods to gain control over endocablings that do \emph{not} come from the subalgebra of $\Endl(\G(X))$ that is given by $\gen{\lambda_z : z \in \Zen(\G(X))}$.
\end{problem}

A disadvantage of our endocabling method is its very limited applicability. By \cref{pro:adding_two_cablings} and \cref{pro:diagonal_of_cabling_by_center}, we have very good control over cablings of the form
\[
\varphi_r = \sum_{z \in \Zen(\G(X))} k_z \lambda_z,
\]
where 
\[
r =  \sum_{z \in \Zen(\G(X))}k_ze_z \in \Z[\Zen(\G(X))],
\]
in particular with respect to the behaviour of the diagonal, which can be computed as
\[
T_{\varphi} = \prod_{z \in \Zen(\G(X))} (T^z)^{k_z}.
\]
From this, it can be shown that the image of the homomorphism
\[
\mathcal{T}: \Z[\Zen(\G(X))] \to \Sym_X; \ r \mapsto T_{\varphi_r}
\]
becomes a module over $\Z[\Zen(\G(X))]$, generated by $T_{\id} = T$, and even inherits a commutative ring structure from $\Z[\Zen(\G(X))]$

However, with any element
\[
r = \sum_{g \in \G(X)}k_ge_g \in \Zen(\Z[\G(X)]
\]
also comes a canonical $\lambda$-en\-do\-mor\-phism given by
\[
\varphi_r = \sum_{g \in \G(X)} k_g \lambda_g
\]
that we do not have any control over (yet), except for the fact that $T_{\varphi_{r+s}} = T_{\varphi_r} \circ T_{\varphi_s}$ for $r,s \in \Zen(\Z[\G(X)]$. This limits applications of endocabling to situations where central elements are a priori available, for example when $\G(X)$ is nilpotent. In other cases, the endocabling method restricts to the ``classical'' cabling method of Lebed, Ram\'{i}rez and Vendramin, as long as there is no further knowledge about other endocablings than those coming from $\Zen(\G(X))$.

\begin{problem}
    Is there a generalization of \cref{thm:full_diagonal_implies_finite_mpl} or \cref{thm:full_diagonal_implies_retractable_2_type} to general integers $n$?
\end{problem}

The methods applied in this work heavily rely on the result that actions of $p$-groups on $p$-groups have non-trivial fixed points, in particular the non-triviality of $\Fix(\G(X))$ if $X$ is of $p$-type, together with the non-triviality of $\Zen(\G(X))$. This leaves open the question if the retractability theorems proven for cycle sets of size $n$ where $T$ is an $n$-cycle can be generalized to cases where $n$ is not a prime power. By the results of Cedó and Okniński \cite{CO_SquarefreeIndecomposable}, an indecomposable cycle set of squarefree size is always retractable. Furthermore, as the case of even $n$ is potentially pathological, due to the potential presence of the exceptional irretractable cycle set $X_{4,19}$ as a (direct) factor, a reasonable first step would be to determine if the theorems generalize to \emph{odd} cardinalities of the form $n = p^2q$. As there is no computational evidence for truth or falsehood in such a case - the smallest relevant cardinality is $n = 45$ - we avoid stating this question as a conjecture.

\section*{Acknowledgements}

The author of this article wants to express his gratitude to Marco Castelli, Edouard Feingesicht and Senne Trappeniers for plenty of helpful discussions and constructive criticism.

\bibliography{references}
\bibliographystyle{abbrv}

\section*{Appendix: Savile Row code\texorpdfstring{ for \cref{lem:start_the_induction}}{}}

\begin{verbatim}
language ESSENCE' 1.0

find C : matrix indexed by [int(0..15),int(0..15)] of int(0..15)

such that

$ Axiom (C1)
forAll i: int(0..15). allDiff([C[i,j] | j: int(0..15)]),
$ Axiom (C3)
forAll i,j,k: int(0..15). C[C[i,j],C[i,k]] = C[C[j,i],C[j,k]],
$ Diagonal condition
forAll i: int(0..15). C[i,i] = ((i+1)%16),
$ x_i -> x_1-i centralizes G(X) 
forAll i,j : int(0..15). C[i,(1-j)%16] = (1-C[i,j])%16,
$ x_i -> x_i+8 is an automorphism
forAll i,j: int(0..15). C[(i+8)%16,(j+8)%16] = (C[i,j]+8)%16,
$ Irretractability
forAll i: int(0..n-2). forAll j: int(i+1..15).
exists k: int(0..15). C[i,k] != C[j,k]
\end{verbatim}

\end{document}